\documentclass[11pt, letterpaper, dvips]{article}
\usepackage{amsmath}
\usepackage{amssymb}
\numberwithin{equation}{section}
\usepackage{amsthm}
\usepackage{url}
\usepackage{geometry}
\allowdisplaybreaks

\newcommand{\matgrp}[3]{\ensuremath{#1^{#2 \times #3}}}
\newcommand{\matring}[2]{\matgrp{#1}{#2}{#2}}
\newcommand{\columnspace}[2]{\matgrp{#1}{#2}{1}}
\newcommand{\rowspace}[2]{\matgrp{#1}{1}{#2}}

\newcommand{\F}{\ensuremath{\textup{\textsf{F}}}}
\newcommand{\K}{\ensuremath{\textup{\textsf{K}}}}

\newcommand{\Z}{\ensuremath{\mathbb{Z}}}

\newtheorem{theorem}{Theorem}[section]
\newtheorem{lemma}[theorem]{Lemma}

\begin{document}
\title{Selecting Algorithms for Black Box Matrices:
  Checking for Matrix Properties That Can Simplify Computations}
\author{Wayne Eberly\\
       Department of Computer Science\\
       University of Calgary\\
       \url{eberly@ucalgary.ca}}
\maketitle

\begin{abstract}
Processes to automate the selection of appropriate algorithms for various matrix computations are described.
In particular, processes to check for, and certify, various matrix properties of black-box matrices are presented.
These include sparsity patterns and structural properties that allow ``superfast'' algorithms to be used in place
of black box algorithms. Matrix properties that hold generically, and allow the use of matrix preconditioning to
be reduced  or eliminated, can also be checked for and certified --- notably including in the small-field case,
where this presently has the greatest impact on the efficiency of the computation.
\end{abstract}

\section{Introduction}
\label{sec:introduction}

\newcommand{\Mult}[1]{\ensuremath{\mathcal{M}(#1)}}

Krylov-based ``black box'' algorithms for matrix computations have been used for significant applications
in computational number theory.  They also form a significant part of the C++ template library
\textsc{LinBox} for high-performance matrix computations. These are notable, in part, because of their
versatility: Any matrix representation that allows the input matrix (or, for some algorithms, its transpose) to
be multiplied by a vector can be supported.

Considerably more efficient algorithms can be used instead if the input matrix is sparse, with nonzero entries
limited to specific locations, or satisfies one of various structural properties. As described, for example, by
Golub and van Loan~\cite{golo12}, Gaussian Elimination can be used quite efficiently to solve banded systems
of linear equations. As described by Pan~\cite{pan01}, various classes of matrices (including Toeplitz-like
and  Hankel-like matrices) have various displacement structures that can be used to reduce system solving
for these matrices to polynomial arithmetic. Under these circumstances, assistance  in selecting algorithms to
be employed might be of help as the community of users of systems like \textsc{LinBox} grows and
non-expert users should be better supported.

Sections~\ref{sec:band_matrix}--\ref{sec:displacement_rank} of this report therefore concern attempts to
detect and certify matrix properties to facilitate algorithm selection. The preliminary results given here
establish that band matrices and matrices with low displacement rank --- including Toeplitz-like and
Hankel-like matrices --- are easily detected and certified. Furthermore, it is possible to convert matrix
representations, in order to allow superfast algorithms to be applied, when such matrices are discovered.

As black-box algorithms have been developed, several matrix properties have been identified that hold
generically and that be exploited --- generally by eliminating ``matrix preconditioning'' --- to simplify or
accelerate computations without sacrificing reliability. In particular, the cost of system solving can be reduced
if the input matrix has a small number of nontrivial nilpotent blocks in its Jordan normal form. Simpler
algorithms to compute the rank or characteristic polynomial of a matrix can be applied if the input
matrix has a small number of nontrivial invariant factors.
Sections~\ref{sec:additive_conditioners}--\ref{sec:invariant_factors} concern the detection and
certification of these properties. A technique of Villard~\cite{vill00} is adapted, for the small-field case, to
efficiently check for these properties at a cost linear in that needed to apply Wiedemann's algorithm to
compute the minimal polynomial of a matrix. Interactive protocols, of the type recently described by Dumas
and Kaltofen~\cite{dum14} are also provided.
\begin{table}
\centering
\begin{tabular}{|l|c|c|c|}
\hline
Property & Detection/Certification & Verification &Communication\\
\hline
Band Matrix & $\mu k$  & $nk + \mu$  & $nk$ \\
Low Displacement Rank & $nk^2 + \mu k \log n$ & $nk + n \log n + \mu$ & $nk + n \log n$ \\
Few Nilpotent Blocks & $n^2 k + \mu n$ & $nk + \mu$ & $nk$ \\
Many Nilpotent Blocks & $n^2 k + \mu n$ & $nk + \mu$ & $nk$ \\
Few Invariant Factors & $n^2 \Mult{n} + n^2 k \log n$ &
  $n\Mult{n} + nk + \mu$ & $nk + n \log n$ \\
  & $\hspace*{50pt} + \mu n \log n$ & & \\
Many Invariant Factors & $n^2 \Mult{n} + n^2 k \log n$ &
     $n \Mult{n} + nk +  \mu$ &  $nk + n \log n$ \\
     & \hspace*{50pt} $+ \mu n \log n$ & & \\
\hline
\end{tabular}
\caption{Summary of Results}
\label{fig:summary}
\end{table}

Of course, many randomized black box algorithms are Las Vegas, so that one can simply execute algorithms
without preconditioning, in hopes that desirable matrix properties are satisfied or that one ``gets lucky''. The
above results may nevertheless be of interest if one considers exchanges between a service provider and client
involving the cost of a service that is to be provided: One would hope here that the cost to the
service provider (or ``prover'') would not exceed the lower cost to carry out a computation without
preconditioning, while the cost to the client (or ``verifier') would be significantly lower than that. Furthermore,
a process to certify that preconditioning is necessary would also be of interest. Protocols to certify this are
also given.

The expected (and, in a few cases, worst-case) costs established to detect and certify these properties,
and for their verification, are linear in the expressions
shown in Table~\ref{fig:summary}. In each case, the indicated cost is the number
of field operations required to carry out the indicated operation
for a black-box matrix $A \in \matring{\F}{n}$. Here, $\mu \le n(2n+1)$ is the number of
operations required to multiply either $A$ or $A^T$ by a given vector $v \in \columnspace{\F}{n}$. It is also
assumed that $\mu \ge n$.  $\Mult{n}$ is the number of operations in~$\F$ required for arithmetic in an
extension $\text{\textsf{E}}$ of~$\F$ with logarithmic degree, so that $\Mult{n} \in O(\log_2 n \log_2 \log_2 n
\log_2 \log_2 \log_2 n)$. The ``communication'' reports the number of elements of the ground field $\F$ (or, in
some cases, bits) that must be communicated between a prover and a verifier --- excluding the initial cost to
communicate a black box matrix~$A \in \matring{\F}{n}$, parameter~$k$ and error tolerance~$\epsilon$.
In typical applications one might expect $k$ to be significantly smaller than~$n$ --- indeed, polylogarithmic.
The cost to check for band structure or low displacement rank, and return the information needed for superfast
algorithms to be applied when they can, is significantly dominated by the cost to use a black box algorithm
to complete a computation instead, in this case.

A conference report~\cite{ebe16}, including the material found in this report but that omits the proof of
Theorem~\ref{thm:probable_relationship}, is also available.

\section{Band Structure}
\label{sec:band_matrix}

Let $A \in \matring{\F}{n}$ and let $k$ be a positive constant. Let us say that $A$ is a \textbf{\emph{band
matrix}} with \textbf{\emph{band width}}~$k$ if the entry $a_{i, j}$ of~$A$ in row~$i$ and column~$j$ is
equal to~$0$ whenever $1 \le i, j \le n$ and $|i - j| > k$. Golub and van Loan~\cite{golo12} describe efficient
algorithms, based on Gaussian Elimination, for computations on such matrices.

As shown below the detection and certification of a black-box matrix that is a band matrix, and conversion to a
representation allowing other algorithms to be used, is surprisingly easy. Indeed, this is included, in part, to
provide a very simple first example.

The objective of this section is to prove the following.

\begin{theorem}
\label{thm:band_matrix_summary}
It is possible to check whether a matrix $A \in \matring{\F}{n}$ is a band matrix, with band width~$k$,  by
selecting $n$ values uniformly and independently from a finite subset~$S$ of~$\F$ and by performing
$\Theta(\mu k)$ arithmetic operations over~$\F$. If $A$ is, indeed, a band matrix, then this is confirmed with
certainty. Otherwise the probability that $A$ is mistaken for a band matrix is at most $1/|S|$.

A certificate with size $\Theta(nk)$ --- which also allows algorithms for band matrix computations to be applied
to~$A$ --- can be supplied when $A$ is a band matrix. This certificate can be verified by selecting $n$ entries
uniformly and independently from a set~$S$, as above, and using $\Theta(nk + \mu)$ arithmetic operations
over~$\F$. Once again if $A$ is, indeed, a band matrix and the certificate is correct, then it is accepted with
certainty. If the certificate is incorrect then it is accepted with probability at most $1/|S|$.
\end{theorem}

\subsection{Detection and Certification}
\label{ssec:band_matrix_discovery}

Let $K = 2k+1$ and, for $1 \le i \le K$, let $\alpha_{K, i} \in \columnspace{\F}{n}$ such that, for $1 \le j \le n$,
the $j^{\text{th}}$ entry of~$\alpha_{K, i}$ is equal to one if $j \equiv i \bmod{K}$ and is zero otherwise.
If $A$ has band width~$k$ then no two (or more) of columns $i, i+K, i+2K, \dots$ of~$A$ have
nonzero entries in the same row. The nonzero entries of these columns can therefore simply be read off as
entries of the vector $A \cdot \alpha_{K, i}$ --- and all of the nonzero entries of~$A$ can be read off the
nonzero entries of each of the vectors $A \cdot \alpha_{K, 1}, A \cdot \alpha_{K, 2}, \dots, A \cdot \alpha_{K, K}$. In particular, if $1 \le j \le n$, then the $j^{\text{th}}$
column of~$A$ can only have nonzero entries in rows $s, s+1, s+2, \dots, t$,
where $s = \max(j-k, 0)$ and $t = \min(j+k, n)$ --- and these entries are the entries
of the vector $A \alpha_{K, j \bmod{K}}$ in positions $s, s+1, s+2, \dots, t$.

Consequently, $K = 2k+1$ multiplications of~$A$ by vectors suffice for the prover  to produce a representation
of $A$ as a band matrix with band width~$k$ if it indeed has this structure.

Of course, this should not be delivered to a verifier without being checked. It is possible that there is no band
matrix $\widehat{A} \in \matring{\F}{n}$, with band width~$k$, such that $A \alpha_{K, i} = \widehat{A}
\alpha_{K, i}$ for $1 \le i \le K$ --- for it may be necessary for a matrix to have off-band entries in some of
columns $k+1, k+2, \dots, K$ or $n-K, n-K+1, \dots, n-k-1$ in order for it to satisfy these equations. In particular
(when $K \le n$), this is the case if the top entry of $A \alpha_i$
is nonzero for any integer~$i$ such that $k+1 \le i \le K$, or if the bottom entry
of~$A \alpha_i$ is nonzero for any integer~$i$ such that $1 \le i \le K$ and
$i \in \{ n-K+1 \bmod{K}, n-K+2 \bmod{K}, \dots, n-k-1 \bmod{k} \}$.

Otherwise the band matrix~$\widehat{A}$ that satisfies these conditions is unique and it suffices to check that
$A = \widehat{A}$. An application of the test of Frievalds~\cite{frievalds79} suffices to check this: A
vector $x \in \columnspace{\F}{n}$, whose entries are chosen uniformly and independently from a finite
subset~$S$ of~$\F$, should be selected by the prover, and it should be checked whether $Ax = \widehat{A} x$.
If this is not the case then $A \not= \widehat{A}$ and, once again, one should stop.

On the other hand, it is easily checked that if $A \not= \widehat{A}$, and $x$ is chosen as above, then the
probability that $Ax = \widehat{A} x$ is at most $1/|S|$ --- so that, provided that $S$ is sufficiently large, the
prover should deliver a certificate so that a verification stage can proceed.

\subsection{Verification}
\label{ssec:band_matrix_verification}

The certificate provided to the verifier, at this point, should simply be a representation of~$A$ as a band
matrix --- presumably provided as an $n \times (2k+1)$ array reporting the entries within the bands of~$A$.

This can be verified using an independent repetition of the Frievalds test described above.

Since the product of a band matrix $A \in \matring{\F}{n}$ (with band width~$k$) and a vector $x \in
\columnspace{\F}{n}$ can be computed using $\Theta(nk)$ field operations and zero tests,
Theorem~\ref{thm:band_matrix_summary} is now immediate --- assuming, again, that $\mu \ge n$.

\section{Low Matrix Rank}
\label{sec:rank}

The following is less general than the protocol of Dumas and Kaltofen~\cite{dum14} to certify matrix rank
and, therefore,  inferior in at least one significant respect. However, it can be used in the special case needed
here: One is certifying that the rank of $A \in \matring{\F}{n}$ is at most~$k$, when $k$ is significantly smaller
than~$n$. It also includes the construction of an alternative representation of~$A$ as needed to support the
claims in Section~\ref{sec:displacement_rank}.

Suppose now that $A \in \matring{\F}{n}$ has positive rank $r \le k$. Then there exist permutation matrices
$P,  Q \in \matring{\F}{n}$, matrices $L \in \matgrp{\F}{(n-r)}{r}$ and $R \in \matgrp{\F}{r}{(n-r)}$, and a
nonsingular matrix $C \in \matring{\F}{r}$, such that
\begin{equation}
\label{eq:low_rank_decomposition}
A = P \times \begin{bmatrix} I_r \\ L \end{bmatrix} \times C \times \begin{bmatrix} I_r & R \end{bmatrix} \times Q.
\end{equation}

The objective of this section is to establish the following.

\begin{lemma}
\label{lem:low_rank_summary}
It is possible to check whether a matrix $A \in \matring{\F}{n}$ has rank at most~$k$, by selecting
$\Theta(nk)$ values uniformly and independently from a finite subset~$S$ of~$\F$ and performing
$(n k^2 + \mu k \log n)$ arithmetic operations in~$\F$ and $\Theta(n \log_2 n)$ operations on bits. This
process fails with probability at most \mbox{$(\min(r, k)+1)/|S|$}, where $r$ is the rank of~$A$, and only by
returning an estimate of the rank of~$A$ that is too low.

If $A$ has rank at most~$k$ then a decomposition of~$A$, as shown at
line~\eqref{eq:low_rank_decomposition}, can be computed at the above cost and returned as a certificate.
This certificate can be verified by choosing $n$ values uniformly and independently from a finite subset~$S$
of~$\F$ and performing $\Theta(nk + \mu)$ arithmetic operations in~$\F$ and $\Theta(n \log_2 n)$ operations
on bits. If the certificate is correct then it is accepted with certainty. Otherwise it is accepted with probability
at most $1/|S|$.
\end{lemma}

\subsection{Detection and Certification}
\label{ssec:low_rank_detection}

Since the rank of~$A$ cannot exceed that of~$C$, no decomposition as shown at
line~\eqref{eq:low_rank_decomposition} can exist unless the rank of~$A$ is at most~$k$. A prover can
check for this condition by attempting to construct the matrices included in this decomposition, along
with~$C^{-1}$.

Let $0 \le \ell \le k$ and suppose indices  $i_1, i_2, \dots, i_{\ell}$ of rows and $j_1, j_2, \dots, j_{\ell}$
of columns of an $\ell \times \ell$ nonsingular submatrix~$C_{\ell}$ of~$A$ have been computed, along with
the matrix $C_{\ell}^{-1}$.

If $\ell = 0$ then the prover should begin by generating $n$ values uniformly and independently from a finite
subset~$S$ of~$\F$ and using these as the entries of a vector $x \in \columnspace{\F}{n}$. If $A$ is nonzero
then $A x \not= 0$ with probability at least $1 - 1/|S|$ --- so that (if $|S|$ is sufficiently large) the prover may conclude
that the rank of~$A$ is zero, if $Ax=0$, and proceed to delivery of a certificate.

Otherwise $x$ should be used to locate a nonzero column of~$A$. Suppose that $x$ has $h$~nonzero
entries. Set $x_1, x_2 \in \columnspace{\F}{m}$ such that $x_1$ has $\lceil h/2 \rceil$ nonzero entries, $x_2$
has $\lfloor h/2 \rfloor$ nonzero entries, and $x = x_1 + x_2$. One should check whether $A x_1 \not= 0$ ---
replacing~$x$ with $x_1$ if this is the case, and replacing $x$ with~$x_2$ otherwise, since
$ A x_2 = A x \not= 0$ in this second case. Iterating this process at most $\lceil \log_2 n \rceil$ times, a
vector~$x \in \columnspace{\F}{n}$ such that $Ax \not= 0$, and $x$ has a single nonzero entry in
some position~$j_1$, is obtained --- establishing that the $j_1^{\text{th}}$ column of~$A$ is nonzero. This
column has now been computed, as $Ax$, and $i_1$ can be chosen to be any integer such that
$1 \le i_1 \le n$ and the entry $\alpha$ of~$A$ in row~$i_1$ and column~$j_1$ is nonzero. Now
\[
 C_1 = \begin{bmatrix} \alpha \end{bmatrix} \quad \text{and} \quad C_1^{-1} = \begin{bmatrix} \alpha^{-1} \end{bmatrix}.
\]

If $\ell > 0$ then the prover should begin, once again, by forming the vector~$x$ as described above.
The prover should continue by computing the matrix-vector product $v = Ax$, and setting $y \in
\columnspace{\F}{\ell}$ to be the vector such that, for $1 \le h \le \ell$, the entry of~$y$ in position~$h$ is the
entry of~$v$ in position $i_h$.

The vector $z = C_{\ell}^{-1} y \in \columnspace{\F}{\ell}$ should next be computed. Let $w \in
\columnspace{\F}{n}$ be the vector such that, for $1 \le h \le \ell$, the entry of~$w$ in position~$j_h$ is the
entry of~$z$ in position~$h$, and such that all other entries of~$w$ are zero. Finally, set $u = v - A \cdot w$ ---
noting that $v$ is in the space spanned by columns $j_1, j_2, \dots, j_{\ell}$ of~$A$ if and only if $u = 0$.

If the rank of~$A$ is equal to~$\ell$ then $u$ must always be equal to zero; $u$ is nonzero
with probability at least $1 - 1/|S|$ otherwise. Consequently if $u = 0$ then the prover should proceed with the
completion of a certificate, as described below.

Otherwise, if $x$ has $h$ nonzero entries then one should once again set $x_1, x_2 \in \columnspace{\F}{n}$
such that $x_1$ has $\lceil h/2 \rceil$ entries, $x_2$ has $\lfloor h/2 \rfloor$ entries, and $x = x_1 + x_2$.
The above process should be applied to~$x_1$ (instead of~$x$) to check whether $A x_1$ is in the space
spanned by columns $j_1, j_2, \dots, j_{\ell}$ of~$A$ --- replacing $x$ with~$x_1$ if this is not the case, and
replacing $x$ with~$x_2$ otherwise. Iterating this process at most $\lceil \log_2 n \rceil$ times one eventually
obtains a vector  $x \in \columnspace{\F}{n}$ such that $Ax$ is not in the space spanned by columns
$j_1, j_2, \dots, j_{\ell}$ of~$A$ and $x$ has a single nonzero entry in some position $j_{\ell + 1}$. This
establishes that the $j_{\ell+1}^{\text{th}}$ column of~$A$ is not in the space spanned by columns $j_1, j_2,
\dots, j_{\ell}$ --- and that columns $j_1, j_2, \dots, j_{\ell+1}$ of~$A$ are linearly independent.

One should next compute the vector $u \in \columnspace{\F}{n}$, as described above, corresponding to the
final choice of the vector~$x$ --- so that $u \not= 0$. It suffices to choose $i_{\ell+1}$ such that $1 \le
i_{\ell + 1} \le n$ and the $i_{\ell+1}^{\text{th}}$ entry of~$u$ is nonzero in order to ensure that the submatrix
$C_{\ell +1}$ of~$A$, including entries in rows $i_1, i_2, \dots, i_{\ell+1}$ and columns $j_1, j_2, \dots,
j_{\ell + 1}$, is nonsingular.

Note next that
\[
 C_{\ell+1} = \begin{bmatrix} C_{\ell} & s \\ t & \alpha \end{bmatrix}
\]
for vectors $s \in \columnspace{\F}{\ell}$ and $t \in \matgrp{\F}{1}{\ell}$, and for some value $\alpha \in \F$.
Since $C_{\ell}$ is nonsingular,
\[
 C_{\ell+1} = \begin{bmatrix} I_{\ell} & 0 \\ t C_{\ell}^{-1} & 1 \end{bmatrix} \cdot
  \begin{bmatrix} C_{\ell} & 0 \\ 0 & \beta \end{bmatrix} \cdot
  \begin{bmatrix} I_{\ell} & C_{\ell}^{-1} s \\ 0 & 1 \end{bmatrix}
\]
where $\beta = \alpha - t C_{\ell}^{-1} s$. Now $\beta \not= 0$, since $C_{\ell+1}$ is also nonsingular, and
\begin{align*}
 C_{\ell+1}^{-1} &= \begin{bmatrix} I_{\ell} & -C_{\ell}^{-1} s \\ 0 & 1 \end{bmatrix} \cdot
  \begin{bmatrix} C_{\ell}^{-1} & 0 \\ 0 & \beta^{-1} \end{bmatrix} \cdot
  \begin{bmatrix} I_{\ell} & 0 \\ -t C_{\ell}^{-1} & 1 \end{bmatrix} \\
  &= \begin{bmatrix} C_{\ell}^{-1} + \left( C_{\ell}^{-1} s \right) \cdot \left( \beta^{-1} t C_{\ell}^{-1} \right) & - C_{\ell}^{-1} s \beta^{-1} \\
    - \beta^{-1} t C_{\ell}^{-1} & \beta^{-1}
  \end{bmatrix}.
\end{align*}
Since $C_{\ell}^{-1} s \in \columnspace{\F}{\ell}$ and $\beta^{-1} t C_{\ell}^{-1} \in \matgrp{\F}{1}{\ell}$, this
expression for~$C_{\ell+1}^{-1}$ can be used to compute the entries of~$C_{\ell+1}^{-1}$ using
$\Theta(\ell^2)$ operations in~$\F$. The value~$\ell$ can now be incremented and the above process
repeated.

If this process is iterated until $\ell = k+1$, then the rank of~$A$ is greater than~$k$ and one can stop.
Otherwise the rank $r \le k$ of~$A$ has been obtained, along with the matrix $C = C_r \in \matring{\F}{r}$
shown at line~\eqref{eq:low_rank_decomposition}, and the indices of the rows and columns of this matrix
in~$A$.

The permutation matrices~$P$ and~$Q$, shown at line~\eqref{eq:low_rank_decomposition}, can each be
concisely represented as an integer vector, with length~$n$, whose $i^{\text{th}}$ entry is the index of the
nonzero entry in row~$i$ of the permutation matrix. Since the first $r$~entries of this representation of~$Q$
are the indices $j_1, j_2, \dots, j_r$, it is not difficult to compute this representation of~$Q$ using $O(n)$
operations on integers whose binary representations have length~$O(\log n)$: The only operations required
are the initialization of these vectors, comparisons of integers and assignments of values. If a first array is
initially sorted and a second integer array is used to maintain the locations of each of $1, 2, \dots, n$ in the
initial array then one can reorder $1, 2, \dots, n$ in order to obtain this representation of~$Q$ using $O(r)$
exchanges of values in this array. The second array, mentioned above, is then a representation of~$Q^T$.
Since the initial entries of a representation of~$P^T$ are the indices
$i_1, i_2, \dots, i_r$, a representation of~$P^T$ can be computed in the same way using $O(n)$ operations
on integers with length in~$O(\log n)$. A representation of $(P^T)^T = P$ is also obtained as a result
of this process.

It remains only to notice that if $A_L \in \matgrp{\F}{n}{r}$ is the matrix including columns $j_1, j_2, \dots, j_r$
of~$A$, and $A_R \in \matgrp{\F}{r}{n}$ is the matrix including rows $i_1, i_2, \dots, i_r$ of~$A$, then 
\[
 \begin{bmatrix}
   I_r \\ L
 \end{bmatrix}
 = P^T \cdot A_L \cdot C^{-1}
 \qquad \text{and} \qquad
 \begin{bmatrix}
   I_r & R
 \end{bmatrix}
 = C^{-1} \cdot A_R \cdot Q^T.
\]
Since $C^{-1}$ has already been computed, $L$ and~$R$ can be computed using $O(n r^2)$ additional
arithmetic operations in~$\F$ and $\Theta(n \log_2 n)$ operations on bits.

A consideration of the above confirms that  $\Theta(nk^2 + k \mu \log n)$ arithmetic operations in~$\F$ and 
$\Theta(n \log_2 n)$ operations on bits have been used, in the worst case, to check whether the rank of~$A$
is at most~$k$, and to compute the rank and the decomposition at line~\eqref{eq:low_rank_decomposition}
if this is the case. This process can only fail due to unlucky choices of the randomly selected vectors
$x \in \columnspace{\F}{n}$, described above. Since each selection fails with probability at most $1/|S|$ and
at most $\min(r, k)+1$ such vectors must be selected if $A$ has rank~$r$, the total probability of failure is
at most $(r+1)/|S|$ if $r \le k$ and at most $(k+1)/|S|$ otherwise.

\subsection{Verification}
\label{ssec:low_rank_verification}

Once again, it suffices to apply the Frievalds test to verify that $A$ is the zero matrix, if the reported rank is
zero, or that the decomposition of~$A$, shown at line~\eqref{eq:low_rank_decomposition}, is correct
otherwise. An examination of this decomposition confirms that this test can be carried out at the cost stated
in the above lemma.

\section{Low Displacement Rank}
\label{sec:displacement_rank}

For $\alpha \in \F$, the $n \times n$ \textbf{\emph{$\boldsymbol{\alpha}$-circulant matrix $\boldsymbol{Z_{\alpha}}$}} is the matrix
\[
 Z_{\alpha} = \begin{bmatrix}
  0 &      &         &     & \alpha \\
  1 & 0 &           &     &   \\
     & 1 &           &     &   \\
     &    & \ddots &    &   \\
     &     &           & 0 &   \\
     &     &           & 1 & 0
  \end{bmatrix} \in \matring{\F}{n}
\]
whose entry in row~$i+1$ and column~$i$ is~$1$ for $1 \le i \le n-1$, whose entry in row~$1$ and column~$n$
is~$\alpha$, and all of whose other entries are zero. Consider the following linear operators on matrices
in~$\matring{\F}{n}$:
\begin{itemize}
\setlength{\itemsep}{0pt}
\item $\varphi_T(A) = Z_1 \cdot A - A \cdot Z_0$.
\item $\varphi_H(A)  = Z_1 \cdot A - A \cdot Z_0^T$.
\item $\varphi_{TH}(A) = (Z_0 + Z_0^T) \cdot A - A \cdot (Z_0 + Z_0^T)$.
\end{itemize}
A matrix $A \in \matring{\F}{n}$ is \textbf{\emph{Toeplitz-like}} (respectively, \textbf{\emph{Hankel-like}}, and
\textbf{\emph{Toeplitz+Hankel-like}}) if the rank of the matrix $\varphi_T(A)$ (respectively, $\varphi_H(A)$, and
$\varphi_{TH}(A)$) is small relative to~$n$. The matrix $\varphi_T(A)$ (respectively, $\varphi_H(A)$
or $\varphi_{TH}(A)$) is called the \textbf{\emph{operator matrix}} and rank of this matrix is said to be the
\textbf{\emph{displacement rank}} of~$A$. As described, for example, by Pan~\cite{pan01}, a variety of matrix
computations have ``superfast algorithms'' if the displacement rank of a matrix is low. Indeed, if the
displacement rank is polylogarithmic in~$n$ then the worst-case running times of these algorithms are
generally within a polylogarithmic factor of linear in~$n$.

A black box for multiplication of $\varphi_T(A)$ (respectively, $\varphi_H(A)$ or $\varphi_{TH}(A)$)  by a vector
is trivially
obtained by applying a black box for multiplication of~$A$ by a vector, twice, and performing $O(n)$ additional
operations in~$\F$. The following, is therefore, immediate from Lemma~\ref{lem:low_rank_summary}.

\begin{theorem}
\label{thm:displacement_rank_summary}
One can check whether a matrix $A \in \matring{\F}{n}$ is Toeplitz-like, Hankel-like, or
Toeplitz+Hankel-like, with displacement rank at most~$k$, and return a representation of the operator matrix
of~$A$ allowing a superfast algorithm to be applied to~$A$ if this is the case.

The cost to check for these properties, produce and return the above representation of the operator matrix,
and verify it --- and the probabilities and types of failures of these processes --- are as described in
Lemma~\ref{lem:low_rank_summary} for the detection, certification and verification of a matrix with
low rank.
\end{theorem}

\section{Additive Conditioners}
\label{sec:additive_conditioners}

Recall that the \textbf{\emph{invariant factors}} of a matrix $A \in \matring{\F}{n}$ are monic polynomials
\[ \varphi_1, \varphi_2, \dots, \varphi_m \in \F[z], \] each with positive degree, such that $\varphi_i$ is divisible
by~$\varphi_{i+1}$ for $1 \le i \le m-1$ and such that $A$ is similar to a block diagonal matrix with the companion
matrices of the polynomials $\varphi_1, \varphi_2, \dots, \varphi_m$ as its blocks. In this case $\varphi_1$ is
the minimal polynomial of~$A$.  An invariant factor $\varphi_i$ is a \textbf{\emph{nontrivial invariant factor}} if
$\varphi_i \not= z$ --- for its companion matrix is different from the $1 \times 1$ zero matrix in this case.
Additional (trivial) ``invariant factors" $\varphi_i = 1$ will occasionally be added, below, for $m+1 \le i \le n$,
to simplify technical statements.

The number of invariant factors divisible by~$z^2$ is of interest because this is the same as the number of
``nontrivial nilpotent blocks'' (companion matrices of polynomials $z^j$ for $j \ge 2$) in a Jordan normal form
for~$A$.

Techniques of Villard~\cite{vill00}  that were developed during the study of a black box algorithm for the
Frobenius normal form  lead to an efficient interactive protocol to bound the number of nontrivial nilpotent blocks.
In combination with a recent protocol of Dumas, Kaltofen, Thom\'{e} and Villard~\cite{dum16} for the certification
of the minimal polynomial of a matrix, these lead to an efficient protocol to bound the number of nontrivial
invariant factors of a matrix as well.

In particular, the following result of Villard~\cite[Lemma~1]{vill00} is of use here.

\begin{theorem}[Villard~\cite{vill00}]
\label{thm:certain_relationship}
Let $A, B \in \matring{\F}{n}$ such that the rank of~$B$ is at most~$k$. If $s_1, s_2, \dots s_n$ are the invariant
factors of~$A$ and $\sigma_1, \sigma_2, \dots, \sigma_n$ are the invariant factors of \mbox{$A + B$} then $s_i$
is divisible by~$\sigma_{i+k}$ in~$\F[z]$ for $1 \le i \le n-k$.
\end{theorem}

Villard also provided a result --- \cite[Theorem 2]{vill00} --- which is of use to confirm that $A$ \emph{does} not
have $k$ or more nontrivial nilpotent blocks, or nontrivial invariant factors, when $\F$ is sufficiently large. The
following result complements Villard's result by allowing this to be checked for, when $k$ is small, and when
$\F$ is a very small finite field --- the case where ``preconditioning'' is generally most complicated and
expensive, so that the assurance that preconditioning can be avoided might be of greatest interest.

\begin{theorem}
\label{thm:probable_relationship}
Let $A \in \matring{\F}{n}$ where $\F$ is a finite field with size~$q$. Let $B = V \cdot U$
where $U \in \matgrp{\F}{k}{n}$, $V \in \matgrp{\F}{n}{k}$, and the entries of~$U$ and~$V$ are selected
uniformly and independently from~$\F$.
\begin{enumerate}
\renewcommand{\labelenumi}{\text{\textup{(\alph{enumi})}}}
\item If $A$ has at most $k$ nontrivial nilpotent blocks then the minimal polynomial of $A + B$ is not divisible
     by~$z^2$ with probability at least
     \[
      \rho_1(q) = \frac{(q^2-2)(q^2-q-1)(q-1)}{q^4 (q+1)} = 1 - \frac{3q^4 + 2q^3 - 5q^2 + 2}{q^5 + q^4}
        \ge 1 - 3 q^{-1}.
     \]
\item If $A$ has at most $k$ nontrivial invariant factors and $f \in \F[z]$ is an irreducible polynomial with 
     degree~$d$ such that $f \not= z$, then the probability that $f$ does not divide the minimal polynomial
     of~$A + B$ is at least
     \begin{align*}
      \rho_2(q, d) &= \frac{(q^{4d} - 2)(q^{2d} - q^d - 1)}{q^{3d} (q^{3d} + q^{2d} + q^d + 1)} \\
      &= 1 - \frac{2 q^{5d} + 2 q^{4d} + q^{3d} + 2 q^{2d} - 2 q^d - 2}{q^{3d}(q^{3d} + q^{2d} + q^d + 1)}
      \ge 1 - 2 q^{-d}.
     \end{align*}
\end{enumerate}
\end{theorem}

The proof of this result is, regrettably, rather long, but is also reasonably straightforward: Basic linear algebra
and probability theory suffice to establish the above result.

\begin{lemma}
\label{lem:rank_one_change}
Let $A \in \matring{\F}{n}$. 

\begin{enumerate}
\renewcommand{\labelenumi}{\text{\textup{(\alph{enumi})}}}
\item If $A$ has rank $r < n$ and the entries of vectors $u \in \rowspace{\F}{n}$ and
     $v \in \columnspace{\F}{n}$ are chosen uniformly and independently from~$\F$ then
     $A + v \cdot u$ has rank $r+1$ with  probability $(1 - |\F|^{-(n-r)})^2$, rank $r-1$ with probability
     at most $|\F|^{-2(n-r)}$, and rank~$r$, otherwise.
     
\item Let $\ell$ be a positive integer. Suppose that $A \in \matring{\F}{n}$ is nonsingular, $v \in
    \matgrp{\F}{n}{1}$,  and that $R \in \matgrp{\F}{\ell}{n}$. Then either
     \begin{enumerate}
     \renewcommand{\labelenumii}{\text{\textup{\roman{enumii}.}}}
     \item $A + v \cdot u \cdot R$ is nonsingular for every vector $u \in \matgrp{\F}{1}{\ell}$, or
     \item if the entries of a vector $u \in \matgrp{\F}{1}{\ell}$ are chosen uniformly and independently
         from~$\F$ (and independently from the entries of~$A$, $v$ and~$R$) then $A + v \cdot u \cdot R$
         is nonsingular with probability $1 - |\F|^{-1}$.
     \end{enumerate}

\item If $A$ is nonsingular and $v = 0 \in \columnspace{\F}{n}$ then $A + v \cdot u$ is nonsingular,
     as well, for every vector $u \in \rowspace{\F}{n}$. If $v$ is a nonzero vector in~$\columnspace{\F}{n}$
     and the entries of $u \in \rowspace{\F}{n}$ are chosen uniformly and independently from~$\F$
     then $A + v \cdot u$ is nonsingular with probability $1 - |\F|^{-1}$.
\end{enumerate}
\end{lemma}

\begin{proof}
Suppose first that $A$ has rank $r < n$ and the entries of vectors $u \in \rowspace{\F}{n}$
and $v \in \columnspace{\F}{n}$ are chosen uniformly and independently from~$\F$.

Since the column space of~$A$ includes $|\F|^{r}$ vectors, $v$ is not in the column space of~$A$
with probability $1 - |\F|^{-(n-r)}$. This is a necessary condition for the rank of $A + v \cdot u$ to exceed
that of~$A$, since the column space of~$A + v \cdot u$ is a subspace of the column space of~$A$,
otherwise.

With that noted, suppose that $v_1$ is not in the column space of~$A$.

Let $w_1, w_2, \dots, w_n \in \columnspace{\F}{n}$ be the columns of the matrix $A \in \matring{\F}{n}$
being considered. Permuting columns as needed we may assume without loss of generality that the first
$r$ columns of~$A$, $w_1, w_2, \dots, w_{r}$, are linearly independent.

Let $\mu_1, \mu_2, \dots, \mu_n$ be the entries of the vector $u \in \rowspace{\F}{n}$. Since columns
$w_1, w_2, \dots, w_{r}$ are linearly independent, and $v$ is not in the column space of~$A$,
it is easily checked that the first $r$ columns 
\[ w_1 + \mu_1 v, w_2 + \mu_2 v, \dots, w_{r} + \mu_{r} v \]
of the matrix $A + v \cdot u$ must be linearly independent as well --- so that the rank of~$A + v \cdot u$ is
at least $r$. On the other hand, the column space of this matrix is a subspace of the space spanned by
$w_1, w_2, \dots, w_{r}, v$, so that the rank of this matrix is also at most $r+ 1$.

Consider any choice of the first $r$ entries, $\mu_1, \mu_2, \dots, \mu_r$, of~$u$, and let $i $ be
an integer such that $r+1 \le i \le n$. Since $A$ has rank $r$ the $i^{\text{th}}$ column~$w_i$ of~$A$ must
be a linear combination of the first $r$ columns, so that there exist elements $\alpha_1, \alpha_2, \dots,
\alpha_{r} \in \F$ such that
\[
  w_i = \sum_{j=1}^{r} \alpha_j w_j.
\]
Now --- again, since $v$ is not in the column space of~$A$ --- it is easily checked that
the $i^{\text{th}}$ column $w_i + \mu_i v$ of the matrix $A + v \cdot u$ is only a linear combination of the first
$r$ columns of this matrix if
\[
 w_i + \mu_i v = \sum_{j=1}^{r} \alpha_j (w_j + \mu_j v)
\]
as well --- for the same values $\alpha_1, \alpha_2, \dots, \alpha_{r} \in \F$ as above. In this case one
can see --- by considering the multipliers for~$v$ in the above equation --- that it must also be true that
\[
 \mu_i = \sum_{j=1}^{r} \alpha_j \mu_j
\]
so that there is only one choice of~$\mu_i$ for which this condition holds. Since the values
$\mu_{r+1}, \mu_{r+2}, \dots, \mu_n$ are chosen uniformly and independently from~$\F$, it now follows
that $A + v \cdot u$ has rank $r$ (instead of $r+1$) with probability $|\F|^{-(n-r)}$ if $v_1$ is not in the
column space of~$A$. Since the entries of~$u$ and~$v$ are chosen uniformly and independently, it
follows that $A + v \cdot u$ has rank $r+1$ with probability $(1 - |\F|^{-(n-r)})^2$, as claimed.

As noted above, the rank of~$A + v \cdot u$ can only be $r-1$ if $v$ is in the column space of~$A$,
and the probability of this is $|\F|^{-(n-r)}$. Virtually the same argument establishes that the rank
of $A + v \cdot u$ can only be $r-1$ if $u$ is in the row space of~$A$, as well, and the probability of
this is also $|\F|^{-(n-r)}$. Since  the entries of~$u$ and~$v$ are chosen independently, the probability
that $A + v \cdot u$ has rank $r-1$ is at most $|\F|^{-2(n-r)}$.

Finally, since the ranks of~$A$ and~$A + v \cdot u$ can differ by at most one, the
rank of $A + v \cdot u$ is in $\{ r-1, r, r+1 \}$, as required
to complete the proof of part~(a) of the claim.

Suppose next that $A$ is nonsingular, $v \in \columnspace{\F}{n}$, $\ell$ is a positive integer, and
$R \in \matgrp{\F}{\ell}{n}$.

If $v = 0$, then $A + v \cdot u \cdot R = A$ for every vector $u \in \rowspace{\F}{\ell}$, so that $A + v \cdot u \cdot R$
is certainly nonsingular as well, and case~(i), mentioned in the claim, holds.

Otherwise $A + v \cdot u \cdot R$ is singular if and only if there is a nonzero vector $x \in \columnspace{\F}{n}$
such that $(A + v \cdot u \cdot R) x = Ax + v \cdot (u \cdot R \cdot x) = 0$. In this case $Ax$ is a nonzero scalar
multiple of~$v$. Now,  since $(A + v \cdot u \cdot R) x = 0$ if and only $(A + v \cdot R \cdot u) (\alpha x) = 0$ for
any nonzero $\alpha \in \F$, it suffices to consider the unique nonzero vector $x = -A^{-1} v$ --- in which case
$(A + v \cdot u \cdot R) x = 0$ if and only if $u \cdot (R \cdot  x) = 1$.

If $R \cdot x = 0 \in \columnspace{\F}{\ell}$, then $(A + v \cdot u \cdot R)x = Ax = -v \not= 0$ for every
vector $u \in \rowspace{\F}{\ell}$, and case~(i) holds once again.

Suppose, instead, that $R \cdot x$ is a nonzero vector in $\columnspace{\F}{\ell}$ and that the entries
of~$u \in \rowspace{\F}{\ell}$ are chosen uniformly and independently from~$\F$ (and independently of
the entries of~$A$, $v$, and~$R$). Consider an integer~$i$ such that $1 \le i \le \ell$ and the $i^{\text{th}}$
entry of~$R \cdot x$ is nonzero. After all other entries of~$u$ have been selected there is exactly one choice
of the $i^{\text{th}}$ entry of~$u$ such that $u \cdot R \cdot x = 1$. Thus $(A - v \cdot R \cdot u)x=0$ with
probability at most~$|\F|^{-1}$, so that case~(ii) holds --- establishing part~(b) of the claim.

Part~(c) of the claim is a trivial consequence of part~(b), obtained by setting $\ell = n$ and setting $R$
to be the identity matrix $I_n \in \matring{\F}{n}$.
\end{proof}

\begin{lemma}
\label{lem:making_nonsingular}
Let $A \in \matring{\F}{n}$ be a matrix with rank $r$ for a nonnegative integer~$r$, and let $k$ be a positive
integer such that $k \ge n-r$.

\begin{enumerate}
\renewcommand{\labelenumi}{\text{\textup{(\alph{enumi})}}}
\item If $k = n - r$, and the entries of matrices $U \in \matgrp{\F}{k}{n}$ and $V \in \matgrp{\F}{n}{k}$
     are chosen uniformly and
     independently from~$\F$, then $A + V \cdot U$ is nonsingular with probability at least
     \[
       \left( 1 - \frac{1}{|\F|^2-1} \right) (1 - |\F|^{-1})^2 = \frac{(|\F|^2 - 2)(|\F|-1)}{|\F|^2 (|\F|+1)}
        \ge 1 - 2 |\F|^{-1}.
     \]
     
\item Let $\ell$ be a positive integer. Suppose that $A \in \matring{\F}{n}$ is nonsingular, $V \in
     \matgrp{\F}{n}{k}$, and $R \in \matgrp{\F}{\ell}{n}$. If the entries of $U \in \matgrp{\F}{k}{\ell}$
     are chosen uniformly and independently from~$\F$ (and independently from the entries of~$A$,
     $V$ and~$R$) then $A + V \cdot U \cdot R$ is nonsingular with probability at least 
     \[
     \frac{|\F|-1}{|\F|} = 1 - |\F|^{-1}
     \]
     if $k = 1$, and with probability at least
     \[
      \left( 1 - \frac{|\F|^{-1}}{|\F|-1} \right) \cdot (1 - |\F|^{-1}) = 1 - \frac{|F|+1}{|\F|^2} > 1 - 2 |\F|^{-1}
     \]
     when $k \ge 2$.

\item If $A$ is nonsingular, $V \in \matgrp{\F}{n}{k}$, and the entries of a matrix $U \in \matgrp{\F}{k}{n}$
     are chosen uniformly and independently from~$\F$ (and independently from the entries of~$A$ and~$V$),
     then the matrix $A + V \cdot U$ is nonsingular with probability at least 
     \[
     \frac{|\F|-1}{|\F|} = 1 - |\F|^{-1}
     \]
     if $k = 1$, and with probability at least
     \[
      \left( 1 - \frac{|\F|^{-1}}{|\F|-1} \right) \cdot (1 - |\F|^{-1}) = 1 - \frac{|F|+1}{|\F|^2} > 1 - 2 |\F|^{-1}
     \]
     when $k \ge 2$.
     
\item If $n-r < k$ and $A$ has rank $r < n$, and the entries of matrices $U \in \matgrp{\F}{k}{n}$
     and $V \in \matgrp{\F}{n}{k}$ are chosen uniformly and independently from~$\F$, then $A + V \cdot U$
     is nonsingular with probability at least
     \[
        \frac{(|\F|^4 - 2)(|\F|^2 - |\F| - 1)}{|\F|^3 (|\F|^3 + |\F|^2 + |\F| + 1)}
         = 1 - \frac{2 |\F|^5 + 2 |\F|^4 + |\F|^3 + 2 |\F|^2 - 2|\F| - 2}{|\F|^3 (|\F|^3 + |\F|^2 + |\F| + 1)}
         \ge 1 - 2 |\F|^{-1}.
     \]
\end{enumerate}
\end{lemma}

\begin{proof}
Suppose first that $A \in \matring{\F}{n}$ is an arbitrarily chosen matrix with rank~$r$
such that $k \ge n-r$. Note that if a matrix $U \in \matgrp{\F}{k}{n}$ has rows
$u_1, u_2, \dots, u_k \in \matgrp{\F}{1}{n}$ (from top to bottom) and $V \in \matgrp{\F}{n}{k}$ has
columns $v_1, v_2, \dots, v_k \in \matgrp{\F}{n}{1}$ (from left to right) then
\[
  A + V \cdot U = A + \sum_{i=1}^k v_i \cdot u_i.
\]
With that noted --- supposing, as above, that the entries of~$U \in \matgrp{\F}{k}{n}$ and
$V \in \matgrp{\F}{n}{k}$ are chosen uniformly and independently from~$\F$ --- let $\rho_{k, r}$
be the probability that there exist integers $j_1, j_2, \dots, j_{n-r}$ such that
\[
 1 \le j_1 < j_2 < \dots < j_{n-r} \le k
\]
and such that the matrix
\[
  A + \sum_{h=1}^{n-r} v_{j_h} \cdot u_{j_h}
\]
is nonsingular. It trivially follows that $\rho_{k, n} = 1$ for every integer~$k$ such that $0 \le k \le n$.

Suppose next that $r < n$. Since the entries of $U \in \matgrp{\F}{k}{n}$ and $V \in \matgrp{\F}{n}{k}$
are chosen uniformly and independently from~$\F$, it follows by a straightforward application of
part~(a) of Lemma~\ref{lem:rank_one_change} that the probability that $A + v_i \cdot u_i$ has rank
at most~$r$, for every integer~$i$ such that $1 \le i \le k$, is $(2 |\F|^{r-n} - |\F|^{2(r-n)})^k$.

Suppose, instead, that there exists an integer~$i$ such that $1 \le i \le k$ and $A + v_i \cdot u_i$
has rank~$r+1$. Permuting the rows of~$U$ and columns of~$V$ as needed --- without changing
the distributions used to generate these matrices --- we may assume without loss of generality that
$i = k$. Now let
\[
  \widehat{A} = A + v_k \cdot u_k \in \matring{\F}{n},
\]
a matrix with rank $r+1$. Since the entries of the vectors $u_1, u_2, \dots, u_{k-1} \in \matgrp{\F}{1}{n}$
and the vectors $v_1, v_2, \dots v_{k-1} \in \matgrp{\F}{n}{1}$ are chosen uniformly and independently
from~$\F$ ---and independently of either the entries of $u_k \in \matgrp{\F}{1}{n}$ and
$v_k \in \matgrp{\F}{n}{1}$, or of the entries of the above matrix~$\widehat{A}$ --- one can now consider
$\widehat{A}$ instead of~$A$ to conclude that
\[
  \rho_{k, r} = \left( 1 - \left(2 |\F|^{r-n} - |\F|^{2(r-n)}\right)^k \right) \cdot \rho_{k-1, r+1}.
\]

It follows, by induction on~$n-r$, that
\begin{align*}
 \rho_{k, r} &= \left( 1 - \left(2 |\F|^{r-n} - |\F|^{2(r-n)}\right)^k \right) \cdot \rho_{k-1, r+1} \\
   &= \prod_{i=0}^{n-r-1} \left( 1 - \left(2 |\F|^{r-n+i} - |\F|^{2(r-n+i)}\right)^{k-i} \right).
\end{align*}

If $0 \le i \le n-r-2$ then
\begin{align*}
 1 - \left(2 |\F|^{r-n+i} - |\F|^{2(r-n+i)}\right)^{k-i}
  &= 1 - |\F|^{(r-n+i)(k-i)} \left(2 - |\F|^{r-n+i}\right)^{k-i} \\
  &\ge 1 - (2|\F|^{r-n+i})^{k-i} \\
  &\ge 1 - (|\F|^{k-i})^{r-n+i+1}. \\
\end{align*}
If $i = n-r-1$ then
\[
 1 - \left(2|\F|^{r-n+i} - |\F|^{2 (r-n+i)}\right)^{k-i} = 1 - (2 |\F|^{-1} - |\F|^{-2})^{k-(n-r) + 1}.
\]

Thus
\begin{align}
 \rho_{k, r} &\ge \left( \prod_{i=0}^{n-r-2} 1 - (|\F|^{k-i})^{r-n+i+1} \right) \cdot 
   (1 - (2|\F|^{-1} - |\F|^{-2})^{k-(n-r)+1}) \notag \\
  &= \left( \prod_{i=0}^{n-r-2} 1 - (|\F|^{i-k})^{n-r-i-1}\right) \cdot 
    (1 - (2|\F|^{-1} - |\F|^{-2})^{k-(n-r)+1}) \notag \\
  &\ge \left( \prod_{i=0}^{n-r-2} 1 - (|\F|^{(n-r) - k - 2})^{n-r - i - 1} \right) \cdot 
    (1 - (2|\F|^{-1} - |\F|^{-2})^{k-(n-r)+1}) \tag{since $(n-r) - k - 2 \ge i-k$ when $0 \le i \le n-r-2$} \\
  &\ge \left( 1 - \sum_{i=0}^{n-r-2} (|\F|^{(n-r) - k - 2})^{r-n - i - 1} \right) \cdot
    (1 - (2|\F|^{-1} - |\F|^{-2})^{k-(n-r)+1})  \notag \\
  &\ge \left( 1 - \sum_{j \ge 1} (|\F|^{(n-r) - k - 2})^{j} \right) \cdot
    (1 - (2|\F|^{-1} - |\F|^{-2})^{k-(n-r)+1}) \notag \\
  \label{eq:probability_of_first_event}
  &= \left( 1 - \frac{1}{|\F|^{k - (n-r) + 2} - 1} \right) \cdot 
    (1 - (2|\F|^{-1} - |\F|^{-2})^{k-(n-r)+1}) 
\end{align}

Suppose, now, that $k = n-r$, so that $r < n$ since $k$ is a positive integer. Then
\[
 1 - (2|\F|^{-1} - |\F|^{-2})^{k-(n-r)+1} = (1 - |\F|^{-1})^2.
\]
It is easily checked (setting $k = n-r$) that this establishes part~(a) of the claim.

Part~(b) follows by a similar argument: Suppose, now, that $A$ is nonsingular, let $k \ge 0$,
$V \in \matgrp{\F}{n}{k}$, $\ell \ge 1$, $R \in \matgrp{\F}{\ell}{n}$,  and suppose that the entries
of~$U \in \matgrp{\F}{k}{\ell}$ are chosen uniformly and independently from~$\F$ (and independently
of the entries of~$A$, $V$ and~$R$). Let $\mu_k$ be the probability that the matrix $A + V \cdot U \cdot R$
is also nonsingular. It trivially follows that $\mu_0 = 1$.

Suppose next that $k \ge 1$. The probability that $A + V \cdot U \cdot R$ is nonsingular can be
under-approximated by the probability that both this is the case and there exists an integer~$i$
such that $1 \le i \le k$ and $A + v_i \cdot u_i \cdot R$ is nonsingular where $u_1, u_2, \dots, u_k
\in \rowspace{\F}{\ell}$ are the rows of~$U$ (from top to bottom) and $v_1, v_2, \dots, v_k
\in \columnspace{\F}{n}$ are the columns of~$v$ (from left to right).

Since the entries of $U \in \matgrp{\F}{k}{\ell}$ are chosen uniformly and independently from~$\F$ (and
independently of the entries of~$A$, $V$ and~$R$), it follows by a straightforward application of part~(b) of
Lemma~\ref{lem:rank_one_change} that $A + v_i \cdot u_i \cdot R$ is singular, for every integer $i$ such that
$1 \le i \le k$, with probability at most $|\F|^{-k}$.

Suppose now that $A + v_i \cdot u_i \cdot R$ is nonsingular for at least one integer~$i$ such that $1 \le i \le k$.
Once again, permuting the rows of~$U$ and columns of~$V$ as needed --- without changing the
distribution to generate these matrices --- we may assume without loss of generality that $i = k$. Let
\[
 \widehat{A} = A + v_k \cdot u_k \cdot R,
\]
a nonsingular matrix in~$\matring{\F}{n}$. Since the entries of the vectors $u_1, u_2, \dots, u_{k-1}
\in \matgrp{\F}{1}{\ell}$ and vectors $v_1, v_2, \dots,v_{k-1} \in \matgrp{\F}{n}{1}$ are chosen uniformly
and independently from~$\F$ --- and independently of either the entries of the vectors
$u_k \in \matgrp{\F}{1}{n}$ and $v_k \in \matgrp{\F}{n}{1}$ or the entries of the above matrix~$\widehat{A}$ ---
we now have that
\[
  \mu_k \ge \left( 1 - |\F|^{-k} \right) \cdot \mu_{k-1}.
\]
Thus $\mu_1 \ge 1 - |\F|^{-1}$ and it is easily established by induction on~$k$ that if $k \ge 2$ then
\begin{align*}
 \mu_k
  &\ge \prod_{i=0}^{k-1} \left( 1 - |\F|^{i-k} \right) \\
  &\ge \left( \prod_{j \ge 2} \left( 1 - |\F|^{-j} \right) \right) \cdot (1 - |\F|^{-1}) \\
  &\ge \left( 1 - \sum_{j \ge 2} |\F|^{-j} \right) \cdot (1 - |\F|^{-1}) \\
  &+ \left( 1 - \frac{|\F|^{-1}}{|\F|-1} \right) \cdot (1 - |\F|^{-1}) \\
  &= 1 - \frac{|\F|+1}{|\F|^2} \\
  &\ge 1 - 2 |\F|^{-1},
\end{align*}
as needed to establish part~(b) of the claim.

Part~(c) of the claim is a trivial corollary of part~(b), obtained by setting $\ell = n$ and setting
$R$ to be the identity matrix $I_n \in \matring{\F}{n}$.

Essentially the same argument (with $\ell = n$ and $R = I_n$) establishes, for $k > n-r$, that if
$A \in \matring{\F}{n}$ has rank~$r$,
then the conditional probability that $A + V \cdot U$ is nonsingular, given that there exist integers
$j_1, j_2, \dots, j_{n-r}$ such that
\[
  1 \le j_1 < j_2 < \dots < j_{n-r} \le k
\]
and the matrix
\[
  A + \sum_{h=1}^{n-r} v_{j_h} \cdot u_{j_h}
\]
is nonsingular, is at least
\begin{equation}
\label{eq:simple_conditional_probability}
1 - |\F|^{-1}
\end{equation}
if $k = n-r+1$, and at least
\begin{equation}
\label{eq:general_conditional_probability}
1 - \frac{|\F|+1}{|\F|^2} \ge 1 - 2 |\F|^{-1}
\end{equation}
if $k \ge n-r+2$.

One can now under-approximate the probability that $A + V \cdot U$ is nonsingular, when
$A$ has rank $r < n$ and $k > n-r$,  by the probability that both this is the case and there exist
integers $j_1, j_2, \dots, j_{n-r}$ such that
\[
  1 \le j_1 < j_2 < \dots < j_{n-r} \le k
\]
and the matrix $A + \displaystyle{\sum_{h=1}^{n-r}} v_{j_h} \cdot u_{j_h}$ is nonsingular. It follows
by the bounds at lines~\eqref{eq:probability_of_first_event} and~\eqref{eq:simple_conditional_probability}
that this is at least $f_1(|\F|)$, where
\begin{align*}
 f_1(z) &= \left( 1 - \frac{1}{z^3-1}\right) \cdot \left( 1 - (2z^{-1}-z^{-2})^2\right) \cdot (1 - z^{-1}) \\
   &= \frac{(z^3-2)}{z(z^2+z+1)} \cdot (1 - (2 z^{-1} - z^{-2})^2)
\end{align*}
if $k = n-r+1$, and --- by the inequalities at lines~\eqref{eq:probability_of_first_event}
and~\eqref{eq:general_conditional_probability} ---  at least $f_2(|\F|)$, where
\begin{align*}
 \left( 1 - \frac{1}{z^4-1} \right) &\cdot \left( 1 - (2z^{-1} - z^{-2})^3 \right) \cdot
    \left( \frac{z^2 - z - 1}{z^2} \right) \\
  &\ge \left( 1 - \frac{1}{z^4-1} \right) \cdot (1 - z^{-1}) \cdot \left( \frac{z^2-z-1}{z^2}\right) \\
  &= \frac{(z^4-2)(z^2-z-1)}{z^3(z^3+z^2+z+1)} = f_2(z)
\end{align*}
if $k \ge n-r+2$.

Suppose first that $|\F|=2$. Then $f_1(|\F|) = \frac{3}{16}$ and $f_2(|\F|) = \frac{7}{60}$, so that
$f_2(|\F|) \le f_1(|\F|)$ in this case.

On the other hand, if $z = |\F| \ge 3$ then
\begin{align*}
f_1(z) &= \frac{(z^3-2)}{z (z^2 + z + 1)} \cdot (1 - (2z^{-1} - z^{-2})^2) \\
 &\ge \frac{(z^3-2)}{z (z^2+z+1)} \cdot (1 - z^{-1}) \\
 &= \frac{(z^3-2)(z-1)}{z^2 (z^2 + z + 1)} = \widehat{f}_1(z).
\end{align*}

Supposing, again, that $z \ge 3$, consider the polynomial
\begin{align*}
 F(z) &= z^3 (z^2+z+1)(z^3+z^2+z+1) \cdot (\widehat{f}_1(z) - f_2(z)) \\
   &= z^6 + 2 z^4 - 2z^2 -2z-2 \in \Z[z].
\end{align*}
Notice that the leading coefficient, $6$, of
\[
 F'(z) = 6 z^5 + 8 z^3 - 4z - 2,
\]
is equal to the sum of the absolute values of all negative coefficients. Since this polynomial includes
a second term with a positive coefficient, $F'(z) > 0$ whenever $z \ge 1$. It therefore suffices to
confirm that $F(3) = 865 > 0$ to confirm that $\widehat{f}_1(|\F|) \ge f_2(|\F|)$ whenever $|\F| \ge 3$.
Thus the desired probability is always at least $f_2(|\F|)$, as needed to establish part~(d) of the claim.
\end{proof}

\begin{lemma}
\label{lem:property_of_charpoly}
If $n$ and~$t$ are positive integers, and $A \in \matring{\F}{n}$ is a matrix with rank $n-t$ such that the
characteristic polynomial of~$A$ is $z^t \varphi(z)$ for a polynomial $\varphi \in \F[z]$ such that
$\varphi(0) \not= 0$, then the minimal polynomial of~$A$ is not divisible by~$z^2$.
\end{lemma}

\begin{proof}
Suppose, to the contrary, that the minimal polynomial of~$A$ is divisible by $z^2$. Then, since $A$ has
rank $n-t$, the $i^{\text{th}}$ invariant factor must be divisible by $z^{j_i}$ for an integer~$j_i$ and for
$1 \le i \le t$, where
\[
 j_1 \ge j_2 \ge \dots \ge j_t \ge 1.
\]
Furthermore $j_1 \ge 2$ since the minimal polynomial of~$A$ is divisible by $z^2$.

Since the characteristic polynomial of~$A$ is the product of the invariant factors of~$A$, it follows that
the characteristic polynomial of~$A$ must be divisible by $\prod_{i=1}^t z^{j_i} = z^{\sum_{i-1}^t j_i}$ and,
since $\sum_{i=1}^t j_i \ge t+1$, the characteristic polynomial of~$A$ cannot be as described in the claim.
\end{proof}

\begin{theorem}
\label{thm:condition_for_nilpotent_blocks}
Let $k$ be a positive integer and let $A \in \matring{\F}{n}$, for a positive integer~$n$, such that $A$ has at most
$k$ nontrivial nilpotent blocks. If the entries of $U \in \matgrp{\F}{k}{n}$ and $V \in \matgrp{\F}{n}{k}$ are
chosen uniformly and independently from~$\F$ then the probability that the minimal polynomial
of $A + V \cdot U$ is not divisible by $z^2$ is at least
\begin{equation}
\label{eq:expression_with_c}
\frac{(|\F|^2-2)(|\F|^2 - |\F| - 1)(|\F|-1)}{|\F|^4 (|\F|+1)}
   = 1 - \frac{3 |\F|^4 + 2 |\F|^3 - 5 |\F|^2 + 2}{|\F|^5 + |\F|^4}
   \ge 1 - 3 |\F|^{-1}
 \end{equation}
\end{theorem}

\begin{proof}
Let $\ell$ be the number of nontrivial nilpotent blocks of~$A$, so that $\ell \le k$. Suppose that $A$ has
rank $r= n - t$ for an integer $t$. Then $\ell \le t \le n$. The cases that $t \le k$ and $t > k$ are
considered, separately, below.

Suppose first that $t \le k$. It follows by parts~(a) and~(d) of Lemma~\ref{lem:making_nonsingular} that
$A + V \cdot U$ is nonsingular with probability at least $\min(f_1(|\F|), f_2(|\F|)$, where
\[
  f_1(z) = \frac{(z^2-2)(z-1)}{z^2(z+1)}
\] 
and
\[
 f_2(z) = \frac{(z^4-2)(z^2-z-1)}{z^3 (z^3+z^2+z+1)}.
\]
The minimal polynomial of $A + V \cdot U$ cannot be divisible by either~$z$ or~$z^2$ if this is the
case.

Suppose next that $t > k$, so that the rank of~$A$ is $n - t < n-k$. In this case --- since the number of
nilpotent blocks with size one in a (rational) Jordan normal form for~$A$ is $t - \ell \ge t - k$ --- $A$ is similar
to a matrix
\begin{equation}
\label{eq:similar_matrix}
 \widetilde{A} = \begin{bmatrix}
   \widehat{A} & 0_{(n-t+k) \times (t-k)} \\
   0_{(t-k) \times (n-t+k)} & 0_{(t-k) \times (t-k)}
 \end{bmatrix},
\end{equation}
where $\widehat{A} \in \matring{\F}{(n-t+k)}$, so that $A = X^{-1} \widetilde{A} X$ for a nonsingular matrix
$X \in \matring{\F}{n}$.

Now, since $A$ and~$\widetilde{A}$ are similar, these matrices have the same rank, invariant factors, and
the same number of nontrivial nilpotent blocks. Furthermore, if the entries of matrices $U \in \matgrp{\F}{k}{n}$
and $V \in \matgrp{\F}{n}{k}$ are chosen uniformly and independently from~$\F$ then so are the entries of the
matrices $\widetilde{U} = U \cdot X^{-1} \in \matgrp{\F}{k}{n}$, and $\widetilde{V} = X \cdot V \in
\matgrp{\F}{n}{k}$. Multiplying by $X$ on the left and by~$X^{-1}$ we may therefore replace~$A$
with~$\widetilde{A}$ --- effectively assuming without loss of generality that $A = \widetilde{A}$ as shown
at line~\eqref{eq:similar_matrix}.

Let $U_L \in \matgrp{\F}{k}{(n-t+k)}$ and $U_R \in \matgrp{\F}{k}{(t-k)}$ be the left and right submatrices
of~$U$, and let $V_T \in \matgrp{\F}{(n-t+k)}{k}$ and $V_B \in \matgrp{\F}{(t-k)}{k}$ be the top and bottom
submatrices of~$V$, so that
\[
 A + V \cdot U = \begin{bmatrix}
   \widehat{A} & 0 \\ 0 & 0 \end{bmatrix} + \begin{bmatrix} V_T \\ V_B \end{bmatrix} \cdot
      \begin{bmatrix} U_L & U_R \end{bmatrix}
   = \begin{bmatrix} A_{1, 1} & A_{1, 2} \\ A_{2, 1} & A_{2, 2} \end{bmatrix}
\]
where
\[ A_{1, 1} = \widehat{A} + V_T \cdot U_L \in \matring{\F}{(n-t+k)},\]
\[ A_{1, 2} = V_T \cdot U_R \in \matgrp{\F}{(n-t+k)}{(t-k)}, \]
\[ A_{2, 1} = V_B \cdot U_L \in \matgrp{\F}{(t-k)}{(n-t+k)}, \]
and
\[ A_{2, 2}  = V_B \cdot U_R \in \matring{\F}{(t-k)}. \]

Since the entries of~$U_L$ and~$V_T$ are chosen uniformly and independently from~$\F$, it follows by
part~(a) of Lemma~\ref{lem:making_nonsingular} (replacing $n$ with $n-t+k$ and replacing $r$ with $n-t$)
that the matrix $A_{1, 1} = \widehat{A} + V_T \cdot U_L$
is nonsingular with probability at least
\[
  \left( 1 - \frac{1}{|\F|^2-1} \right) \cdot (1 - |\F|^{-1})^2 = \frac{(|\F|^2-2)(|\F|-1)}{|\F|^2 (|\F|+1)}.
\]

With that noted, consider (for the rest of this argument) the case that $A_{1,1 }$ is, indeed, nonsingular.

In this case, since the rank of $A + V \cdot U$ cannot exceed $n-t+k$, it follows that there exist matrices
$Y \in \matgrp{\F}{(n-t+k)}{(t-k)}$ and $Z \in \matgrp{\F}{(t-k)}{(n-t+k)}$ such that $A_{1, 2} = A_{1, 1} \cdot Y$,
and $A_{2, 1} = Z \cdot A_{1, 1}$ and --- since $A + V \cdot U$ and $A_{1, 1}$ have the same rank --- $A_{2, 2}
= Z \cdot A_{1, 1} \cdot Y$. Thus
\begin{align*}
 A + V \cdot U
   &= \begin{bmatrix} A_{1, 1} & A_{1, 1} \cdot Y \\ Z \cdot A_{1, 1} & Z \cdot A_{1, 1} \cdot Y \end{bmatrix} \\
   &= \begin{bmatrix} I_{n-t+k} & 0 \\ Z & I_{k-t} \end{bmatrix} \cdot
     \begin{bmatrix} A_{1, 1} & 0 \\ 0 & 0_{k-t} \end{bmatrix} \cdot
     \begin{bmatrix} I_{n-t+k} & Y \\ 0 & I_{k-t} \end{bmatrix}.
\end{align*}
Since the rightmost matrix shown in the above line is nonsingular, $A + V \cdot U$ is similar to the matrix
\begin{align*}
&\begin{bmatrix} I_{n-t+k} & Y \\ 0 & I_{k-t} \end{bmatrix} \cdot
 \begin{bmatrix} I_{n-t+k} & 0 \\ Z & I_{k-t} \end{bmatrix} \cdot
 \begin{bmatrix} A_{1, 1} & 0 \\ 0 & 0_{k-t} \end{bmatrix} \\
 &\hspace*{0.25 true in} = \begin{bmatrix} I_{n-t+k} & Y \\ 0 & I_{k-t} \end{bmatrix} \cdot
   \begin{bmatrix} A_{1, 1} & 0 \\ Z A_{1, 1} & 0_{k-t} \end{bmatrix} \\
 &\hspace*{0.25 true in}
    = \begin{bmatrix} (I_{n-t+k} + YZ) \cdot A_{1, 1} & 0 \\ Z A_{1, 1} & 0_{k-t} \end{bmatrix}.
\end{align*}
The characteristic polynomial of~$A+UV$ is, therefore, the product of $z^{k-t}$ and the
characteristic polynomial
of the matrix $(I_{n-t+k} + YZ) \cdot A_{1, 1}$. Since $A_{1, 1}$ is nonsingular, the matrix
$(I_{n-t+k} + Y Z) A_{1, 1}$  is also nonsingular if $I_{n-t+k} + YZ$ is --- and it would then follow by
Lemma~\ref{lem:property_of_charpoly} (with $t$ replaced by~$t-k$) that the minimal polynomial of
$A + V \cdot U$ is not divisible by~$z^2$.

It now suffices to note that, since $Y = A_{1, 1}^{-1} \cdot A_{1, 2} =
A_{1, 1}^{-1} \cdot V_T \cdot U_R$ and $Z = A_{2, 1} \cdot A_{1, 1}^{-1} = V_B \cdot U_L \cdot
A_{1, 1}^{-1}$, $ I_{n-t+k} + YZ = I_{n-t+k} + \widehat{V} \cdot \widehat{U}\cdot \widehat{R}$ for the matrices
\[
 \widehat{V} = A_{1, 1}^{-1} V_T \in \matgrp{\F}{(n-t+k)}{k},
\]
\[
 \widehat{U} = U_R \in \matgrp{\F}{k}{(t-k)},
\]
and
\[
 \widehat{R} = V_B \cdot U_L \cdot A_{1, 1}^{-1} \in \matgrp{\F}{(t-k)}{(n-t+k)}.
\]
Since the entries of $\widehat{U} = U_R$ are chosen uniformly from~$\F$, and independently from the
entries of~$A$, $U_L$, $V_T$ and~$V_B$, the entries of~$\widehat{U}$ are also chosen independently
from those of~$\widehat{V}$ and~$\widehat{R}$. Since $I_{n-t+k}$ is a fixed nonsingular matrix, it follows
by part~(b) of Lemma~\ref{lem:making_nonsingular}  that the conditional probability that
$I_{n-t+k} + \widehat{V} \cdot \widehat{U} \cdot \widehat{R}$ is nonsingular,
given that $A_{1, 1}$ is, is at least
\[
 \left( 1 - \frac{|\F|^{-1}}{|\F|-1} \right) \cdot (1 - |\F|^{-1}) = \frac{\F|^2 - |\F|-1}{|\F|^2},
\]
so that the minimal polynomial of $A + V \cdot U$ is not divisible by~$z^2$, in this case, with probability
at least $f_3(|\F|)$ where
\[
 f_3(z) = \frac{(z^2 - 2) (z^2 - z - 1) (z-1)}{z^4 (z+1)}.
\]
Now consider the polynomial
\[
 F_1(z) = z^3 \cdot (z+1) \cdot (z^3 + z^2 + z + 1) \cdot (f_1(z) - f_2(z)) = z^4 + z^3 - 2z - 2 \in \Z[z]
\]
and its derivative,
\[
 F'_1(z) = 4 z^3 + 3 z^2 - 2.
\]
Since the leading coefficient, $4$, of~$F'_1$ is greater than the sum of the absolute values of the
negative coefficients of this polynomial, $F'_1(z) > 0$ whenever $z \ge 1$. Since $F_1(2) = 18 > 0$,
it follows that $F_1(z) > 0$ when $z \ge 2$, and that $f_1(z) \ge f_2(z)$ when $z \ge 2$ as well.

Consider as well the polynomial
\begin{align*}
 F_2(z) &= z^4 \cdot (z^3 + z^2 + z + 1) \cdot (z+1) \cdot (f_2(z) - f_3(z)) \\
   &= z^7 + z^6 - 3 z^5 - 3z^4 - z^3 + z^2 + 4z + 2
\end{align*}
and its derivatives
\[
 F_2'(z) = 7 z^6 + 6 z^5 - 15 z^4 - 12 z^3 - 3 z^2 + 2z + 4,
\]
\[
 F_2''(z) = 42 z^5 + 30 z^4 - 60 z^3 - 36 z^2 - 6z + 2,
\]
and
\[
 F_2'''(z) = 210 z^4 + 120 z^3 - 180 z^2 - 72 z - 6.
\]
The leading pair of coefficients of~$F_2'''$ are both positive and their sum, $360$, exceeds the sum
of the absolute values of the negative coefficients of this polynomial. It follows that $F_2'''(z) > 0$
when $z \ge 1$.

Since $F_2''(2) = 1190 > 0$, it now follows that $F_2''(z) > 0$ when $z \ge 2$. Since $F_2'(2) = 300 > 0$,
it also follows that $F_2'(z) > 0$ when $z \ge 2$. Finally, since $F_2(2) =54 > 0$, it follows that
$F_2(z) > 0$ --- and $f_2(z) \ge f_3(z)$ --- whenever $z \ge 2$ as well.

Thus
\begin{align*}
 f_3(|\F|) &= \frac{(|\F|^2 - 2) (|\F|^2 - |\F| - 1)(|\F|-1)}{|\F|^4 (|\F|+1)} \\
   &= 1 - \frac{3 |\F|^4 + 2 |\F|^3 - 5 |\F|^2 + 2}{|\F|^5 + |\F|^4} \\
   &\ge 1 - \frac{3 |\F|^4 + 2 |\F|^3}{|\F|^5 + |\F|^4} \\
   &\ge 1 - 3 |\F|^{-1}
\end{align*}
is a lower bound for the probability that the minimal polynomial of~$A + V \cdot U$ is not divisible
by~$z^2$, in all cases, as needed to establish the claim.
\end{proof}

The next lemma and theorem generalize parts~(a) and~(c) of Lemmas~\ref{lem:rank_one_change},
as well as parts~(a), (c) and~(d) of \ref{lem:making_nonsingular}, as needed to establish
the second part of Theorem~\ref{thm:probable_relationship}.

\begin{lemma}
\label{lem:extended_rank_one_change}
Let $A \in \matring{\F}{n}$ and let $f \in \F[z]$ be a monic irreducible polynomial with degree~$d$. Suppose
the entries of vectors $v \in \columnspace{\F}{n}$ and $u \in \rowspace{\F}{n}$ are chosen uniformly and independently from~$\F$.
\begin{enumerate}
\renewcommand{\labelenumi}{\text{\textup{(\alph{enumi})}}}
\item If $\ell$ of the invariant factors of~$A$ are divisible by~$f$, for a positive integer~$\ell$, then exactly
     $\ell-1$ of the invariant factors of~$A + v \cdot u$ are divisible by~$f$ with probability at least
     $(1 - |\F|^{-\ell d})^2$, and exactly $\ell+1$ of the invariant factors of~$A + v \cdot u$ are divisible by~$f$
     with probability at most $|\F|^{-2 \ell d}$. Exactly $\ell$ of the invariant factors of~$A$ are divisible
     by~$f$, otherwise.
\item If the minimal polynomial of~$A$ is not divisible by~$f$ then the minimal polynomial of $A + v \cdot u$
      is also not divisible by~$f$ with probability at least $1 - |\F|^d$.
\end{enumerate}
\end{lemma}

\begin{proof}
Suppose first that $d = 1$; then $f = z - \lambda$ for some element~$\lambda$ of~$\F$. In this case,
for $\ell \ge 1$, exactly $\ell$ of the invariant factors of~$A$ (respectively, $A + v \cdot u$) are divisible
by~$f$ if and only $A - \lambda I_n$ (respectively, $A - \lambda I_n + v \cdot u$) has rank $r = n - \ell$.
Furthermore, the minimal polynomial of~$A$ (respectively, $A + v \cdot u$) is not divisible by~$f$ if and only if
$A - \lambda I_n$ (respectively, $A - \lambda I_n + v \cdot u$) is nonsingular. The above claims are, therefore,
consequences of parts~(a) and~(c) of Lemma~\ref{lem:rank_one_change} in this case.

Suppose next that $d \ge 2$. Let $\K = \F[\lambda] \cong \F[x]/\langle f \rangle$, where $\lambda \in K$ is a root
of~$f$ in~$\K$. Suppose, as well, that
\[
 f = z^d - \gamma_{d-1} z^{d-1} - \gamma_{d-1} z^{d-2} - \dots - \gamma_1 z - \gamma_0
\]
for $\gamma_{d-1}, \gamma_{d-2}, \dots, \gamma_1, \gamma_0 \in \F$.

In order to establish part~(a) of the claim in this case, recall that every matrix $A \in \matring{\F}{n}$ is similar
to a matrix in ``rational Jordan form'' --- that is, a matrix whose blocks are the companion matrices of powers
of irreducible polynomials in~$\F[z]$. In particular, $A$ is similar to such a matrix, where the first $\ell$ blocks
are the companion matrices of polynomials $f^{j_1}, f^{j_2}, \dots, f^{j_{\ell}}$, for integers
$j_1, j_2, \dots j_{\ell}$ such that
\[
 j_1 \ge j_2 \ge \dots \ge j_{\ell} \ge 1,
\]
and whose remaining blocks are the companion matrices of polynomials that are relatively prime with~$f$.

Now, if $X \in \matring{\F}{n}$ is a nonsingular matrix then the entries of $u \cdot X \in \rowspace{\F}{n}$ and
$X^{-1} \cdot v \in \columnspace{\F}{n}$ are chosen uniformly and independently from~$\F$ if the entries of
$u \in \rowspace{\F}{n}$ and $v \in \columnspace{\F}{n}$ are. Applying a similarity transformation we may
therefore assume without loss of generality that $A$ is in rational Jordan form and, in particular, has the
form described above.

Consider a uniformly selected vector $v \in \columnspace{\F}{n}$ --- noting that this can be written as
\begin{equation}
\label{eq:decomposition_of_v}
 v = \begin{bmatrix}
   v_1 \\ v_2 \\ \vdots \\ v_{\ell} \\ v_{\ell + 1}
  \end{bmatrix}
\end{equation}
where $v_i \in \columnspace{\F}{d \cdot {j_i}}$ for $1 \le i \le \ell$ and where $v_{\ell + 1} \in
\columnspace{\F}{m}$ for $m = n - d \cdot \sum_{i=1}^{\ell} j_i$. In this case, the entries of the vectors
$v_1, v_2, \dots, v_{\ell + 1}$ are selected uniformly and independently from~$\F$.

Each of the first $\ell$ blocks, $C_{f^{j_i}}- \lambda I_{d \cdot j_i}$ of $A - \lambda I_n$ (for $1 \le i \le \ell$)
has nullity one, while the remaining blocks of this matrix are nonsingular.

Let $i$ be an integer such that $1 \le i \le \ell$ and suppose that
\[
 v_i = \begin{bmatrix} \alpha_0 \\ \alpha_1 \\ \vdots \\ \alpha_{d \cdot j_i - 1} \end{bmatrix} \in
 \matgrp{\F}{d \cdot j_i}{1}.
\]
The first  $d \cdot j_i - 1$ columns of $C_{f^{j_i}} - \lambda I_{d \cdot j_i}$ are linearly independent --- indeed,
the $h^{\text{th}}$ column has the nonzero entry $-\lambda$ in position~$h$, $1$ in position $h+1$, and
zeroes everywhere else. The above vector~$v_i$ is therefore in the column space of
$C_{f^{j_i}} - \lambda I_{d \cdot j_i}$  if and only if it is a $\K$-linear combination of these columns. Applying
Gaussian Elimination (and considering entries of this vector from bottom to top) one can confirm
that this is the case if and only if
\[
 \alpha_{d \cdot j_i - 1} + \alpha_{d_{j_i} - 2} \lambda^{-1} + \dots + \alpha_1 \lambda^{-(d \cdot j_i - 2)} +
  \alpha_0 \cdot \lambda^{-(d_i \cdot j_i - 1)} = 0,
\]
which is the case if and only if $g(\lambda) = 0$ for the polynomial
\[
 g = \sum_{h=0}^{d \cdot j_i - 1} \alpha_h z^h \in \F[z].
\]
This is the case if and only if $g$ is divisible by~$f$. Since the coefficients of~$g$ are chosen uniformly
and independently from~$\F$, the probability of this is at most $|\F|^{-d}$.

Since $A - \lambda I_n$ is a block diagonal matrix with
$C_{f^{j_1}} - \lambda \cdot I_{d \cdot j_1}, C_{f^{j_2}} - \lambda \cdot  I_{d \cdot j_2}, \dots,
C_{f_{d_{j_{\ell}}} - \lambda \cdot I_{d \cdot j_{\ell}}}$ as the initial $\ell$ blocks on its diagonal, and the entries
of $v_1, v_2, \dots, v_{\ell+1}$ (and corresponding components of the vector~$u$) are chosen uniformly
and independently from~$\F$, it now follows that $v$ is in the column space of $A - \lambda I_n$ with probability at
most $|\F|^{-\ell d}$.

The matrices $A$ and~$A^T$ have the same invariant factors, and rational Jordan form.
Applying the above argument to~$A^T$, one can see that the probability that $u$ is in the row space
of $A - \lambda I_n$ is at most $|\F|^{-\ell d}$ as well.

As argued in the proof of Lemma~\ref{lem:rank_one_change}, it is necessary for $v$ to be in the column
space of $A - \lambda I_n$ and for $u$ to be in the row space of $A - \lambda I_n$ in order for the rank of
$A - \lambda I_n + v \cdot u$ to be less than that of~$A - \lambda I_n$, and the rank of $A - \lambda I_n
+ v \cdot u$ cannot be less than $n - \ell -1$.  The number of invariant factors of~$A$ therefore divisible by~$f$ is $\ell+1$
with probability at most~$|\F|^{-2 \ell d}$.
There are never more than $\ell +1$ invariant factors of this matrix that are divisible by~$f$.

Suppose, once again, that $v$ is not in the column space of $A - \lambda I_n$, so that the null space of
$A - \lambda I_n + v \cdot u$ is a subspace of the null space of $A - \lambda I_n$. It follows that the null space
of $A - \lambda I_n + v \cdot u$ is a proper subset of the null space of~$A - \lambda I_n$, so that $A + v \cdot u$
has at most $\ell - 1$ invariant factors divisible by~$f$, if and only if there exists a vector $x \in
\columnspace{\F}{n}$ such that $(A - \lambda I_n) x = 0 \not= (A - \lambda I_n + v \cdot u) x = v \cdot (u \cdot x)$.
This is the case if and only if $u \cdot x \not= 0$.

If $A$ is as described above, then it suffices to consider (as~$x$) $\ell$~vectors
\begin{equation}
\label{eq:nullspace_vectors}
\widehat{x}_1 =  \begin{bmatrix}
  x_1 \\ 0 \\ 0 \\ \vdots \\ 0 \\ 0
 \end{bmatrix}, \thinspace
\widehat{x}_2 =  \begin{bmatrix}
  0 \\ x_2 \\ 0 \\ \vdots \\ 0 \\ 0
 \end{bmatrix}, \thinspace
\widehat{x}_3 =  \begin{bmatrix}
  0 \\ 0 \\ x_3 \\ \vdots \\ 0 \\ 0
 \end{bmatrix},
 \dots, \thinspace
 \widehat{x}_{\ell} = \begin{bmatrix}
  0 \\ 0 \\ 0 \\ \vdots \\ x_{\ell} \\ 0
 \end{bmatrix}
\end{equation}
where $x_i$ is a nonzero vector in $\columnspace{\K}{d \cdot j_i}$ such that $(C_{f^{j_i}} - \lambda I_{d \cdot j_i}) x_i = 0$,
for $1 \le i \le \ell$.

Now suppose that $1 \le i \le \ell$ and that
\[
 f^{j_i} = z^{d \cdot j_i} - \zeta_{d \cdot j_i - 1} \lambda^{d \cdot j_i - 1} - \zeta_{d \cdot j_i - 2}
  \lambda^{d \cdot j_i - 2} - \dots - \zeta_1 \lambda - \zeta_0.
\]
Since $f$ is irreducible with degree $d \ge 2$, $z$ does not divide either~$f$ or $f^{j_i}$, so that
$\zeta_0 \not= 0$.

Now
\[
 (C_{f^{j_i}} - \lambda I_{d \cdot j_i}) x_i =
  \begin{bmatrix}
   -\lambda        &   & & \dots & 0 & \zeta_0 \\
   1 & -\lambda  &   & \dots & 0 & \zeta_1 \\
   0 & 1 & -\lambda & \dots & 0 & \zeta_2 \\
   \vdots & \vdots & \vdots & \ddots & \vdots & \vdots \\
   0 &  0 & 0 & \dots & -\lambda & \zeta_{d \cdot j_i - 2} \\
   0 &  0 & 0 & \dots & 1 & \zeta_{d \cdot j_i - 1} - \lambda
  \end{bmatrix}
  \cdot x_i = 0.
\]
The top left submatrix of $C_{f^{j_i}} - \lambda I_n$ with order $d \cdot j_i -1$ is nonsingular --- it is lower
triangular with the nonzero value $-\lambda$ at each diagonal position. The bottom entry of~$x_i$ must
therefore be nonzero. It now suffices to consider a vector
\[
 x_i = \begin{bmatrix}
  \lambda^{d \cdot j_i - 2} \alpha_0 \\
  \lambda^{d \cdot j_i - 3} \alpha_1 \\
  \vdots \\
  \lambda  \alpha_{d \cdot j_i - 3} \\
  \alpha_{d \cdot j_i - 2} \\
  \lambda^{d \cdot j_i - 1}
 \end{bmatrix}
\]
for $\alpha_0, \alpha_1, \dots, \alpha_{d \cdot j_i - 2} \in \K$. Examining the bottom entries in the above
vector, it is now possible to prove inductively that
\begin{align*}
\alpha_i &= \lambda^{d \cdot j_i} - \zeta_{d \cdot j_i - 1} \lambda^{d \cdot j_i - 1}
 - \zeta^{d \cdot j_i - 2} \lambda^{d \cdot j_i - 2}
 - \dots - \zeta_{i+2} \lambda^{i+2} - \zeta_{i+1} \lambda^{i+1} \\
  &= \zeta^i \lambda^i + \zeta^{i-1} \lambda^{i-1} + \dots + \zeta_1 \lambda + \zeta_0.
\end{align*}
Consequently
\[
 x_i =
 \begin{bmatrix}
   \lambda^{d \cdot j_i - 2} \zeta_0 \\
   \lambda^{d \cdot j_i - 2} \zeta_1 + \lambda^{d \cdot j_i - 3} \zeta_0 \\
  \vdots \\
   \lambda^{d \cdot j_i - 2} \zeta_{d \cdot j_i - 2}  + \dots + \lambda \zeta_1 + \zeta_0 \\
  \lambda^{d \cdot j_i - 1}
 \end{bmatrix}
\]
It follows (considering the powers of~$\lambda$ with coefficient~$\zeta_0$, above) that if the entries of
the vector $u \in \rowspace{\F}{n}$ are chosen uniformly and independently from~$\F$ then
$u \cdot \widehat{x}_i = g(\lambda)$ where $g$ is a uniformly chosen polynomial in~$\F[z]$ with degree
at most $d \cdot j_i - 1$. Consequently, since $g(\lambda) = 0$ only if $g$ is divisible by~$f$,
$(A - \lambda I_n + v \cdot u) \widehat{x}_i = 0$ with probability at most $|\F|^{-d}$.
Furthermore,
the events that $(A - \lambda I_n + v \cdot u) \widehat{x}_1 = 0, (A - \lambda I_n + v \cdot u) \widehat{x}_2 = 0,
\dots, (A - \lambda I_n + v \cdot u) \widehat{x}_{\ell} = 0$ are mutually independent, since these involve
pairwise disjoint subsets of the entries of~$u$.

It follows that the conditional probability that $A + v \cdot u$
has exactly $\ell - 1$ invariant factors divisible by~$f$, given that $v$ is not in the column space
of~$A - \lambda I_n$, is at least $(1 - |\F|^{-\ell d})$. Thus $A + v \cdot u$ has exactly $\ell - 1$ invariant factors
divisible by~$f$ with probability at least $(1 - |\F|^{-\ell d})^2$, as needed to complete the proof of part~(a)
of the claim.

In order to prove part~(b), suppose that the minimal polynomial of~$A$ is not divisible by~$f$, so that the
matrix $A - \lambda I_n$ is a nonsingular matrix in~$\matring{\K}{n}$. In this case there exist matrices
$B_0, B_1, \dots, B_{d-1} \in \matring{\F}{n}$ such that
\begin{equation}
\label{eq:A_inverse}
 (A - \lambda I_n)^{-1} = B_0 + \lambda B_1 + \lambda^2 B_2 + \dots + \lambda^{d-1} B_{d-1}.
\end{equation}
It follows that
\begin{align*}
I_n &= (A - \lambda I_n) (B_0 + \lambda B_1 + \lambda^2 B_2 + \dots + \lambda^{d-1} B_{d-1}) \\
 &= A B_0 + \gamma_0 B_{d-1} + \sum_{i=1}^{d-1} (A B_i - B_{i-1} + \gamma_i B_{d-1}) \lambda^i,
\end{align*}
so that $A B_0 + \gamma_0 B_{d-1} = I_n$ and $B_{i-1} = A B_i + \gamma_i B_{d-1}$ for $1 \le i \le d-1$. It
can now be proved inductively, using the above equations, that
\begin{equation}
\label{eq:B_(d-i)}
B_{d-i} =  A^{i-1} B_{d-1} +\sum_{j=d-i+1}^{d-1} \gamma_j A^{j-d+i-1} B_{d-1}
\end{equation}
for every integer~$i$ such that $1 \le i \le d$.

Now let $v \in \columnspace{\F}{n}$. Suppose that $f$ divides the minimal polynomial of $A + v \cdot u$, so
that the matrix $A - \lambda I_n + v \cdot u$ is singular in~$\matring{\K}{n}$.  There must exist a nonzero
vector $x \in \columnspace{\K}{n}$such that $(A - \lambda I_n + v \cdot u) x = 0$. In this case $(A - \lambda I_n) x =
-v \cdot (u \cdot x)$, so that $(A - \lambda I_n) x$ is a $\K$-linear multiple of~$v$. Since $(A - \lambda I_n
+ v \cdot u)x = 0$ if and only if $(A - \lambda I_n + v \cdot u)(\alpha x) = 0$ for any nonzero element $\alpha$
of~$\K$, it now suffices to consider the vector $x = - (A - \lambda I_n)^{-1} v$ --- so that
$u (A - \lambda I_n)^{-1} v = 1$.

Since $(A - \lambda I_n)^{-1}$ is as shown at line~\eqref{eq:A_inverse}, above, it now follows that
$u_i B_0 v = 1$ and $u_i B_i v = 0$ for $1 \le i \le d-1$. It now follows by the equation at line~\eqref{eq:B_(d-i)}
that $u A^j B_{d-1} v = 0$ for $0 \le j \le d-2$ and that $u A^{d-1} B_{d-1} v = 1$.

Consider the minimal polynomial of the matrix~$A$ and the vector $B_{d-1} v$, for $B_{d-1}$ as above --- that
is, the monic polynomial $g \in F[z]$ with least degree such that $g(A) B_{d-1} = 0 \in \columnspace{\F}{n}$.
If the degree of this polynomial is less than~$d$ then there is no vector $u \in \rowspace{\F}{n}$ such the
above conditions are satisfied --- for $A^{d-1} B_{d-1} v$ is a linear combination of $B_{d-1} v, A B_{d-1} v,
\dots, A^{d-2} B_{d-1} v$ in this case, and $u A^{d-1} B_{d-1} v = 0$ if $A^i B_{d-1} v = 0$ for $0 \le i \le d-2$.

On the other hand, if the degree of this minimal polynomial is at least~$d$ then the vectors $B_{d-1} v,
A B_{d-1} v, \dots, A^{d-1} B_{d-1} v$ are linearly independent, so that the matrix $C \in \matgrp{\F}{n}{d}$
with these vectors as columns has maximal rank~$d$. It now suffices to note that the vector
$u \in \rowspace{\F}{n}$ only satisfies the condition required above if
\[
 u C = \begin{bmatrix} 0 & 0 & \dots & 0 & 1 \end{bmatrix} \in \rowspace{\F}{d}.
\]
Since the entries of~$u$ are chosen uniformly and independently from those of~$v$ it now follows that the
probability that $(A - \lambda I_n + v \cdot u) x = 0$ --- and that $A - \lambda I_n + v \cdot u$ is singular --- is
at most $|\F|^{-d}$, establishing part~(b) of the claim.
\end{proof}

\begin{theorem}
\label{thm:condition_for_invariant_factors}
Let $A \in \matring{\F}{n}$ and let $f \in \F[z]$ be a monic irreducible polynomial with degree~$d$ such that at most $k$~invariant
factors of~$A$ are divisible by~$f$ for some positive integer~$k$.
Suppose the entries of matrices $U \in \matgrp{\F}{k}{n}$ and $V \in \matgrp{\F}{n}{k}$ are chosen
uniformly and independently from~$\F$.
\begin{enumerate}
\renewcommand{\labelenumi}{\text{\textup{(\alph{enumi})}}}
\item If exactly $k$ of the invariant factors of~$A$ are divisible by~$f$ then the probability that the minimal
      polynomial of~$A + V \cdot U$ is not divisible by~$f$ is at least
      \[
        \left( 1 - \frac{1}{|\F|^{2d} - 1} \right) (1 - |\F|^{-d})^2 = \frac{(|\F|^{2d} - 2)(|\F|^d-1)}{|\F|^{2d}(|\F|^d+1)}
        \ge 1 - 2 |\F|^{-d}.
      \] 
  
\item If the minimal polynomial of~$A$ is not divisible by~$f$ then the probability that the minimal polynomial
     of~$A + V \cdot U$ is divisible by~$f$ is at least
      \[
       \frac{|\F|^d-1}{|\F|^d} = 1 - |\F|^{-d}
       \]
       if $k = 1$, and with probability at least
      \[
         \left( 1 - \frac{|\F|^{-d}}{|\F|^d-1} \right) (1 - |\F|^{-d}) = 1 - \frac{|\F|^d+1}{|\F|^{2d}}
         > 1 - 2|\F|^{-d}
      \] 
      if $k \ge 2$.
 
 \item The minimal polynomial of $A + V \cdot U$ is not divisible by~$f$ with probability at least
  \begin{align*}
    \frac{(|\F|^{4d} - 2)(|\F|^{2d} - |\F|^d - 1)}{|\F|^{3d} (|\F|^{3d} + |\F|^{2d} + |\F|^d + 1)}
      &= 1 - \frac{2|\F|^{5d} + 2|\F|^{4d} + |\F|^{3d} + 2|\F|^{2d} - 2|\F|^d - 2}
         {|\F|^{3d}(|\F|^{3d} + |\F|^{2d} + |\F|^d + 1)} \\
      &\ge 1 - 2|\F|^{-d}.
  \end{align*}
\end{enumerate}
\end{theorem}

\begin{proof}
The proof of this result is virtually identical to the proof of Lemma~\ref{lem:making_nonsingular}. Rather
than considering the rank of a sequence of matrices, one should consider the number of invariant factors
of a sequence of matrices that are divisible by the polynomial~$f$.
Lemma~\ref{lem:extended_rank_one_change} replaces Lemma~\ref{lem:rank_one_change} in the argument,
so that $|\F|$ is consistently replaced by~$|\F|^d$ in the bounds that are being applied and derived.
\end{proof}

Theorem~\ref{thm:probable_relationship} now follows by
Theorem~\ref{thm:condition_for_nilpotent_blocks} and part~(c)
of~Theorem~\ref{thm:condition_for_invariant_factors}.

\section{Nontrivial Nilpotent Blocks}
\label{sec:nilpotent_blocks}

Recall that a nilpotent block in the Jordan form of a matrix $A \in \matring{\F}{n}$ is \textbf{\emph{nontrivial}}
if has order at least two --- so that the minimal polynomial of this block is $z^j$ for $j \ge 2$. The purpose of this
section is to establish the following.

\begin{theorem}
\label{thm:nilpotent_blocks_summary}
It is possible to decide whether a matrix $A \in \matring{\F}{n}$ has at most $k$ nontrivial nilpotent blocks,
in such a way that the incorrect decision is reached with probability at most~$\epsilon$ for any positive
constant~$\epsilon$. The cost to do so includes the selection of $\Theta(nk)$ values uniformly and
independently from~$\F$ and $\Theta(n^2 k + \mu n)$ arithmetic operations in~$\F$.

It is also possible both to certify that $A$ has at most $k$ nontrivial nilpotent blocks and to certify that $A$ has
more than $k$ nontrivial nilpotent blocks. In both cases the verifier is guaranteed to accept if the prover's
information is correct. The verifier accepts with probability at most~$\epsilon$ if the prover's information is
incorrect.

For both protocols, the expected cost for the prover to complete the protocol is dominated by the worst-case
cost for the initial decision stage, as given above. The cost for the verifier, when confirming that $A$ has at
most $k$ nontrivial nilpotent blocks, includes the selection of $\Theta(n)$ values, uniformly and independently
from~$\F$, and $\Theta(nk  + \mu)$ arithmetic operations in~$\F$. The cost for the verifier, when confirming
that $A$ has more than $k$ nontrivial nilpotent blocks, includes the selection of $\Theta(nk)$ values, uniformly
and independently from~$\F$, and $\Theta(nk + \mu)$ arithmetic operations in~$\F$.
\end{theorem}

\subsection{Detection}
\label{ssec:nilpotent_block_detection}

Let $\sigma_1(2) = 17$, $\sigma_1(3) = 3$, $\sigma_1(4) = \sigma_1(5) = 2$, $\sigma_1(q) = 1$ when
$q=7$ and $\sigma_1(q) = 2$ when $q \ge 8$. It is easily checked that if $\rho_1(q)$ is as given in
Theorem~\ref{thm:probable_relationship}, for every prime power~$q$, then $(1 - \rho_1(q))^{\sigma_1(q)} \le
1/2$ when $2 \le q \le 7$ and $(1 - \rho_1(q))^{\sigma_1(q)} \le q^{-1}$ when $q \ge 8$.

In order to check whether a black-box matrix $A \in \matring{\F}{n}$ has at most $k$ nontrivial
invariant factors, when $\F = \F_q$, and to ensure that the probability of an incorrect decision is at most a
positive constant~$\epsilon$, the prover should select $\tau$ pairs of matrices $U_i \in \matgrp{\F}{k}{n}$
and $V_i \in \matgrp{\F}{n}{k}$, for $1 \le i \le \tau$, by choosing the entries of these matrices uniformly
and independently from~$\F$ --- where $\tau = \lceil \log_2 (2 \cdot \epsilon^{-1}) \rceil \cdot \sigma_1(q)$ if
$2 \le q \le 7$ and $\tau = \lceil \log_q (2 \cdot \epsilon^{-1}) \rceil \cdot \sigma_1(q)$ if $q \ge 8$.

Suppose first that $A$ has more than $k$ nontrivial nilpotent blocks, so that the $k+1^{\text{st}}$ invariant factor of~$A$
is divisible by~$z^2$. Then it follows by Theorem~\ref{thm:certain_relationship}, above, that the minimal
polynomial of $A + V_i \cdot U_i$ is divisible by~$z^2$ for all~$i$.

For every such matrix it is easily checked, in this case, that if $u_{i, j}, v_{i, j} \in \columnspace{\F}{n}$ are
chosen uniformly and independently from $\columnspace{\F}{n}$ then the minimal polynomial of the linearly
recurrent sequence
\begin{equation}
\label{eq:Krylov_sequences}
 u_{i, j}^T v_{i, j},\, u_{i, j}^T (A + V_i \cdot U_i) v_{i, j},\, u_{i, j}^T (A +V_i \cdot U_i)^2 v_{i, j},\,  \dots
\end{equation}
is also divisible by~$z^2$ with probability at least $(1 - 1/q)^2$ --- which is equal to $1/4$ if $q = 2$, and
greater than $1 - 2/q$ if $q \ge 3$. Consequently, if $\lambda = \lceil \log_{4/3} (2  \cdot \tau \cdot \epsilon^{-1})
\rceil$ when $q = 2$, and $\lambda = \lceil \log_{q/2} (2 \cdot \tau \cdot \epsilon^{-1}) \rceil$ when $q \ge 3$, and
pairs of vectors $u_{i, j}$ and~$v_{i, j}$ are chosen uniformly and independently from~$\columnspace{\F}{n}$, 
for $1 \le i \le \tau$ and $1 \le j \le \lambda$, then, for each~$i$, the probability there is no integer~$j$ such that
$1 \le j \le \lambda$, and the minimal polynomial of the linear recurrence at line~\eqref{eq:Krylov_sequences}
is divisible by~$z^2$, is at most $\epsilon/(2\tau)$. The probability that it has not been confirmed that the
minimal polynomial of $A + V_i \cdot U_i$ is divisible by~$z^2$, for all~$i$ such that $1 \le i \le \tau$, is
therefore certainly at most~$\epsilon/2 < \epsilon$ in this case.

It follows that --- for fixed~$q$ and~$\epsilon$ --- the number of applications of Wiedemann's algorithm needed
to compute  the minimal polynomials of sequences as above and confirm the above condition, with the desired
reliability, is a constant.

On the other hand, a straightforward calculation (involving $\tau$, as given above) and an application of
Theorem~\ref{thm:probable_relationship}(a) establishes that if $A$ has at most $k$ nontrivial invariant factors then
the minimal polynomial of at least one matrix $A + V_i \cdot U_i$, such that $1 \le i \le \tau$, is \emph{not}
divisible by~$z^2$ with probability at least $1 - \epsilon/2$.

One should again try to compute the minimal polynomial of each matrix $A + V_i \cdot U_i$ by
computing the minimal polynomials of linearly recurrent sequences of the form shown at
line~\eqref{eq:Krylov_sequences} for $\lambda$ uniformly and independently pairs of vectors
$u_{i, j}, v_{i, j} \in \columnspace{F}{n}$.

If $1 \le i \le \tau$ and the minimal polynomial of $A + V_i \cdot U_i$ is
not divisible by~$z^2$, then the minimal polynomial of the linear recurrent sequence shown at
line~\eqref{eq:Krylov_sequences} is not divisible by~$z^2$, either, for any~$j$ that $1 \le j \le \lambda$.

On the other hand, it follows by the above analysis that there will exist an integer~$j$ such that $1 \le j \le
\lambda$ and the minimal polynomial of the above linearly recurrent sequence is divisible by~$z^2$, for
every integer~$i$ such that $1 \le i \le \tau$ and the minimal polynomial of $A + V_i \cdot U_i$ is divisible
by~$z^2$, with probability at least $1 - \epsilon/2$.

In this case a pair of matrices $U \in
\matgrp{\F}{k}{n}$ and~$V \in \matgrp{\F}{n}{k}$, such that the minimal polynomial of~$A + V \cdot U$ is not
divisible by~$z^2$, can be selected by choosing any one of the pairs of matrices $U_i$ and~$V_i$
that have not been eliminated using the above process. The probability that either this case has not been
correctly identified, or a pair of matrices~$U$ and~$V$ as described above has not been correctly
selected, is at most~$\epsilon$.

Since the cost to multiply $A + V \cdot U$ by a vector~$v \in \columnspace{\F}{n}$, includes the cost, $\mu$,
to multiply~$A$ by a vector, along with $\Theta(nk)$ additional operations, it is straightforward to modify the
analysis of Wiedemann's algorithm in order to conclude that the cost of the above process includes the uniform
and independent selection of $\Theta(nk)$ values from~$\F$, along with $\Theta(n^2 k + \mu n)$ arithmetic
operations in~$\F$, as claimed.

\subsection{Few Nilpotent Blocks: Certification and Verification}
\label{ssec:few_nilpotent_blocks}

If the prover has determined that $A$ has at most $k$ nontrivial nilpotent blocks, as described above,
then the prover should \textbf{\emph{commit}} by sending matrices $U \in \matgrp{\F}{k}{n}$ and $V \in
\matgrp{\F}{n}{k}$, such that the minimal polynomial of $A + V \cdot U$ is not divisible by~$z^2$, to the verifier:
These have now been obtained.

It is easily checked, by consideration of a rational Jordan form, that if the minimal polynomial of a matrix $B 
\in \matring{\F}{n}$ is not divisible by~$z^2$, then a system $Bx = b$ is consistent (for a given vector $b \in
\columnspace{\F}{n}$), if and only if the system $B^2 x = b$ is consistent as well. On the other hand, if the
minimal polynomial of~$B$ is divisible by~$z^2$ and a vector $c \in \matring{\F}{n}$ is selected uniformly
and independently,  then the probability that the system $B^2 x = Bc$ is consistent is at most $|\F|^{-1}$.

The verifier may therefore form a \textbf{\emph{challenge}} by selecting $\gamma = \lceil \log_q \epsilon^{-1}\rceil$ vectors $c_1, c_2, \dots, c_\gamma$ uniformly and independently
from~$\columnspace{\F}{n}$ (for $q = |\F|)$, computing $b_i = (A + V \cdot U) c_i$ for $1 \le i \le \gamma$,
and sending~$b_1, b_2, \dots, b_{\gamma}$ to the prover. 

The prover should compute vectors $x_1, x_2, \dots, x_{\gamma} \in \columnspace{\F}{n}$ such that
$(A + V \cdot U)^2 x_i = b_i$ (possibly by applying Wiedemann's algorithm twice, for each~$i$) and send
these to the verifier. Finally, the verifier should check whether the required equalities are satisfied ---
\textbf{\emph{accepting}} if they are, and \textbf{\emph{rejecting}} otherwise.

If the prover's information is correct then the verifier accepts with certainty; otherwise the verifier accepts,
incorrectly, with probability at most $\epsilon$. The cost to the prover, to complete this protocol, is dominated
by the cost of the ``detection'' stage described above. Since the verifier must only choose $\Theta(n)$~values
uniformly and independently from~$\F$ and multiply a constant number of vectors by $A + V \cdot U$, the
number of operations used by the verifier is as claimed.

\subsection{Many Nilpotent Blocks: Certification and Verification}
\label{ssec:many_nilpotent_blocks}

If the prover has determined, instead, that $A$ has more than $k$ nontrivial nilpotent blocks, then the prover
should \textbf{\emph{commit}} by advising the verifier of this.

The verifier should then select matrices $U_i \in \matgrp{\F}{k}{n}$ and $V_i \in \matgrp{\F}{n}{k}$ for $1 \le i
\le \tau$, for $\tau$ as above, by selecting the entries of these matrices uniformly and independently from~$\F$.
These matrices should then be sent to the prover as a \textbf{\emph{challenge}}.

In \textbf{\emph{response}} the prover should return vectors $x_i \in \columnspace{\F}{n}$ such that
$(A + V_i \cdot U_i) x_i \not= 0 = (A + V_i \cdot U_i)^2 x_i = 0$ for $1 \le i \le \tau$. The verifier should then
\textbf{\emph{accept}} if these conditions are all satisfied and \textbf{\emph{reject}} otherwise.

Once again, Theorem~\ref{thm:certain_relationship} can be used to establish that this protocol is perfectly
complete --- the verifier accepts with certainty if the prover's information is correct.
Theorem~\ref{thm:probable_relationship} and a straightforward calculation establishes that it is also sound:
If the prover's information is incorrect then the probability that the verifier accepts is at most~$\epsilon$.

In order to see that the additional cost for the prover is low, recall that the \textbf{\emph{minimal polynomial}} of
a matrix~$B \in \matring{\F}{n}$ and vector~$v \in \columnspace{\F}{n}$ is the monic polynomial $g \in \F[z]$
with least degree such that $g(B) v = 0$. Wiedemann's algorithm can be used to compute the minimal
polynomial of~$A + V_i \cdot U_i$ and a vector~$v$, as the least common multiple of a number of linear
recurrent sequences as shown at line~\eqref{eq:Krylov_sequences}, with $v = v_{i, j}$ and uniform and
independent choices of the vector~$u_{i, j}$. A small number of choices of~$u_{i, j}$ suffice to ensure that the
minimal polynomial of~$(A + V_i \cdot U_i) v$ has been discovered with high probability. Furthermore, if
vectors $v_{i, j} \in \columnspace{\F}{n}$ are chosen uniformly and independently from~$\columnspace{\F}{n}$,
for $j = 1, 2, \dots$, and the minimal polynomial of $A + V_i \cdot U_i$ is divisible by~$z^2$,
then the expected number of vectors~$v_{i, j}$ that must be considered, before a vector~$v = v_{i, j}$ is found
such that the minimal polynomial of~$A + V_i \cdot U_i$ and~$v$ is also divisible by~$z^2$, is at most two.

Suppose now that a vector~$v \in \columnspace{\F}{n}$ has been discovered and it has been confirmed that
the minimal polynomial of $A + V_i \cdot U_i$ and~$v$ is $z^2 g$ for some polynomial $g \in \F[z]$. It suffices
to compute and return the vector $x_i = g(A + V_i) v$ in order to satisfy the requirements given above.

The expected cost for the prover to complete this protocol is, once again, dominated by the worst-case cost
of the ``detection'' stage. The cost for the verifier includes the selection of $\Theta(nk)$ values uniformly and
independently from~$\F$ and a constant number of multiplications of $A + V_i \cdot U_i$ by vectors, for
matrices $U_i \in \matgrp{\F}{k}{n}$ and $V_i \in \matgrp{\F}{n}{k}$ --- establishing the above claim.

\section{Nontrivial Invariant Factors}
\label{sec:invariant_factors}

The purpose of this section is to establish the following.

\begin{theorem}
\label{thm:invariant_factors_summary}
One can decide whether a matrix $A \in \matring{\F}{n}$ has at most $k$ nontrivial invariant factors, such 
that the incorrect decision is made with probability at most~$\epsilon$ for any positive constant~$\epsilon$.
The expected cost of this includes the selection of $\Theta(nk)$ values uniformly and independently from~$\F$
and $\Theta(n^2 k + \mu n)$ arithmetic operations in~$\F$.

It is also possible both to certify that $A$ has at most $k$ nontrivial invariant factors and to certify that $A$ has
more than $k$ nontrivial invariant factors. In both cases the verifier is guaranteed to accept if the prover's
information is correct. The verifier accepts with probability at most~$\epsilon$ if the prover's information is
incorrect.

For both protocols, the expected cost for the prover to complete the protocol is in \[ O(n^2 \Mult{n} +
n^2 k \log_2 n + \mu n \log_2 n ), \] where $\Mult{n}$ is the number of operations in~$\F$ required for an
arithmetic operation in a extension~$\text{\textup{\textsf{E}}}$ with degree in~$O(\log_2 n)$ over~$\F$.
When certifying that $A$ has at most $k$ nontrivial invariant factors,  the verifier selects $O(n \log_2 n)$ values
uniformly and independently from~$\F$ and performs $\Theta(n \Mult{n} + nk + \mu)$ arithmetic operations
in~$\F$. When certifying that $A$ has more than $k$ nontrivial invariant factors the verifier selects
$O(n \log_2 n + nk)$ values uniformly and independently from~$\F$ and performs $O(n \Mult{n} + nk + \mu)$
arithmetic operations in~$\F$.
\end{theorem}

\subsection{Detection}
\label{ssec:invariant_factors_detection}

The $k+1^{\text{st}}$ invariant factor~$\varphi_{k+1}$ of~$A$ is divisible by~$z^2$ if and only if~$A$ has more
than $k$ nontrivial nilpotent blocks. The process described in Subsection~\ref{ssec:nilpotent_block_detection}
should be applied to check this, with parameters chosen to ensure that the probability of failure is at most
$\epsilon/2$.

If $\varphi_{k+1}$ is not divisible by~$z^2$ then a pair of matrices $U_0 \in \matgrp{\F}{k}{n}$ and $V_0 \in
\matgrp{\F}{n}{k}$ have been found such that the minimal polynomial~$f_0 \in \F[z]$ of $A + V_0 \cdot U_0$ is
not divisible by~$z^2$ --- and $f_0$ has been correctly computed --- with probability at least $1- \epsilon/2$.
The detection stage should proceed with an attempt to compute a factor $\chi \not= z$ of~$\varphi_{k+1}$ with
positive degree, along with a certificate of this factor --- or to determine that no such factor exists.

Suppose first that $q = |\F| = 2$. In this case a straightforward variant of the protocol described in
Subsection~\ref{ssec:nilpotent_block_detection} can be used either to conclude that $\varphi_{k+1}$ is divisible
by~$z+1$ or to obtain matrices~$U_1 \in \matgrp{\F}{k}{n}$ and $V_1 \in \matgrp{\F}{n}{k}$ such that
the minimal polynomial~$f_1 \in\F[z]$ of $A + V_1 \cdot U_1$ is not divisible by~$z+1$. Suppose that this is
carried out in such a way that the probability of failure is at most~$\epsilon/6$ --- so that $U_1$, $V_1$
and~$f_1$ have also been obtained with probability at least $1 - \epsilon/6$ as well in the second case. If
$\sigma_2(2, 1) = 6$ and $\rho_2(q, d)$ is as shown in part~(b) of Theorem~\ref{thm:probable_relationship}
then $(1 - \rho_2(2, 1))^{\sigma_2(2, 1)} < \frac{1}{2}$. It therefore suffices to choose $\tau_2(2, 1) =
\lceil \log_2 (12/\epsilon) \rceil \cdot \sigma_2(2, 1)$ pairs of matrices $\widehat{U}_i \in \matgrp{\F}{k}{n}$
and $\widehat{V}_i \in \matgrp{\F}{n}{k}$, for $1 \le i \le \tau_2(2, 1)$, in order to ensure that the minimal
polynomial of $A + \widehat{V}_i \cdot \widehat{U}_i$ is not divisible by~$z+1$, for at least one of these pairs
of matrices~$\widehat{U}_i$ and~$\widehat{V}_i$, with probability at least $1 - \epsilon/12$, if $\varphi_{k+1}$
is not divisible by~$z+1$. If $\lambda_2(2, 1) = \lceil \log_{4/3} (12 \cdot \tau_2(2, 1) \cdot \epsilon^{-1}) \rceil$
then a consideration of the minimal polynomials of $\lambda_2(2, 1)$ linearly recurrent sequences with the form
shown at line~\eqref{eq:Krylov_sequences}, for $1 \le i \le \tau_2(2, 1)$, suffice to ensure that it has been
correctly discovered whether $z+1$ divides the minimal polynomial of $A + \widehat{V}_i \cdot \widehat{U}_i$,
for all of these matrices, with probability at least $1 - \epsilon/12$ --- as is necessary and sufficient here.

If it has been determined that $\varphi_{k+1}$ is not divisible by~$z+1$, then a variant of the above protocol
should be applied, once again, either to conclude that $\varphi_{k+1}$ is divisible by $z^2 + z + 1$ (the only
monic irreducible polynomial in~$\F[z]$ with degree two)  or to obtain matrices $U_2 \in \matgrp{\F}{k}{n}$ and
$V_2 \in \matgrp{\F}{n}{k}$ such that the minimal polynomial~$f_2 \in \F[x]$  of $A + V_2 \cdot U_2$ is not
divisible by \mbox{$z^2 + z +1$}. Suppose, as above, that this is carried out in such a way that the probability of
failure (including failure to compute $U_2$, $V_2$ and~$f_2$) is at most $\epsilon/6$. Now, if
$\rho_2(q, d)$ is as given in part~(b) of Theorem~\ref{thm:probable_relationship} then
$1 - \rho_2(2, 2) < \frac{1}{2}$. It therefore suffices to choose $\tau_2(2, 2) = \lceil \log_2 (12/\epsilon) \rceil$
pairs of matrices $\widehat{U}_i \in \matgrp{\F}{k}{n}$ and $\widehat{V}_i \in \matgrp{\F}{n}{k}$, for
$1 \le i \le \tau_2(2, 2)$, in order to ensure that the minimal polynomial of $A + \widehat{V}_i \cdot \widehat{U}_i$
is not divisible by~$z^2 + z + 1$, for at least one of these pairs of matrices $\widehat{U}_i$ and~$\widehat{V}_i$,
with probability at least $1 - \epsilon/12$ if $\varphi_{k+1}$ is not divisible by~$z^2 + z + 1$.
 
Suppose now that $\widehat{U}_i \in \matgrp{\F}{k}{n}$ and $\widehat{V}_i \in \matgrp{\F}{n}{k}$ such that the
minimal polynomial of \mbox{$A + V_i \cdot U_i$} is divisible by~$z^2 + z + 1$. Then, if vectors $u_{i, j}, v_{i, j}$
are chosen uniformly and independently from~$\columnspace{\F}{n}$, then the minimal polynomial of the linear
recurrence resembling that shown at line~\eqref{eq:Krylov_sequences} (with $\widehat{U}_i$
and~$\widehat{V}_i$ replacing $U_i$ and~$V_i$, respectively) is divisible by~$z^2 + z + 1$ with probability at
least $(1 - \frac{1}{4})^2 = \frac{9}{16}$. Consequently, if $\lambda_2(2, 2) = \lceil \log_{16/7} (12 \cdot
\tau_2(2, 2) \cdot \epsilon^{-1}) \rceil$, then a consideration of $\lambda_2(2, 2)$ linearly recurrent sequences
resembling the one at line~\eqref{eq:Krylov_sequences}, for $1 \le i \le \tau_2(2, 2)$, suffices to ensure that
it has been correctly discovered whether $z^2 + z + 1$ divides the minimal polynomial of~$A + \widehat{V} \cdot
\widehat{U}$, for all of these matrices, with probability at least $1 - \epsilon/12$ --- as is necessary and
sufficient, here, once again.

Suppose, now, that it has been determined that $\varphi_{k+1}$ is not divisible by $z^2 + z + 1$ either.
Suppose that an additional two pairs of matrices $U_i \in \matgrp{\F}{k}{n}$ and $V_i \in \matgrp{\F}{n}{k}$ are
selected uniformly and independently, for $3 \le i \le  4$. Then, since $1 - \rho_2(2, d) \le 2^{1-d}$ for $d \ge 3$,
and there are only two monic irreducible polynomials in~$\F[z]$ with degree three, while there are at most
$q^d/d = 2^d/d$  monic irreducible polynomials with degree~$d$ in~$\F[z]$ when $d \ge 4$, it follows that if
$f_i$ is the minimal polynomial of~$A + V_i \cdot U_i$, for $3 \le i \le 4$, then the probability that
$\gcd(f_3, f_4)$ has a monic irreducible factor, with degree at least three, that is not also a factor
of~$\varphi_{k+1}$, is at most
\begin{align*}
 2 &\times (1 -\rho_2(2, 3))^2 + \sum_{d \ge 4} \left( \frac{2^d}{d} \right) \cdot (1 - \rho_2(2, d))^2 \\
  &\le 2 \times \left( \frac{1}{4} \right)^2 + \sum_{d \ge 4} \left( \frac{2^d}{d}\right) \cdot \left( \frac{2}{2^d} \right)^2
   \tag{by the bounds in part~(b) of Theorem~\ref{thm:probable_relationship}} \\
   &\le \frac{1}{8} + \sum_{d \ge 4} \left( \frac{2^d}{4} \right) \cdot \left( \frac{2}{2^d} \right)^2 \\
    &= \frac{1}{8} + \sum_{d \ge 4} 2^{-d} \\
    &= \frac{1}{8} + \frac{1}{8} \\
    &= \frac{1}{4}.
\end{align*}

Consequently  if one chooses matrices $U_i \in \matgrp{\F}{k}{n}$ and $V_i \in \matgrp{\F}{n}{k}$ uniformly
and independently, for $3 \le i \le \tau + 2$, instead,  where $\tau = 2 \lceil \log_4 (12 \epsilon^{-1} \rceil$, and
$f_i$ is the minimal polynomial of $A + V_i \cdot U_i$ for all such~$i$, then $\gcd(f_3, f_4, \dots,
f_{\tau + 2})$ has an irreducible factor, with degree at most three, that is not also a factor of~$\varphi_{k+1}$,
with probability at most $\epsilon/12$. Indeed, two pairs of matrices $U_i \in \matgrp{\F}{k}{n}$ and
$V_i \in \matgrp{\F}{n}{k}$ will have been identified, with probability at least $1 - \epsilon/12$, such that if $f_i$ is
the minimal polynomial of~$U + V_i \cdot U_i$, for $3 \le i \le 4$ and $f_0$, $f_1$ and~$f_2$ are as above, then
the squarefree part of $\gcd(f_0, f_1, f_2, f_3, f_4)$ is a divisor of~$\varphi_{k+1}$.

As discussed in Subsection~\ref{ssec:nilpotent_block_detection}, and above, it is possible to ensure that each
of the above minimal polynomials $f_i$ of $A = V_i \cdot U_i$ is computed in such a way that the probability of
failure here is also at most $\epsilon/12$, by computing the minimal polynomials of $\Theta(\log_2
(\epsilon^{-1}))$ linearly recurrent sequences as shown at line~\eqref{eq:Krylov_sequences}. At this point,
either a divisor of~$\varphi_{k+1}$ with positive degree that is different from~$z$ has been identified, or
matrices~$U_i \in \matgrp{\F}{k}{n}$, $V_i \in \matgrp{\F}{n}{k}$, and the minimal polynomials $f_i \in \F[z]$ of
$A + V_i \cdot U_i$ have been identified, for $0 \le i \le 4$, such that $\gcd(f_0, f_1, f_2, f_3, f_4) =
\varphi_{k+1} \in \{1, z\}$. The probability of failure of this process is at most~$\epsilon$, and (for
fixed~$\epsilon$) the prover has selected $\Theta(nk)$ values uniformly and independently from~$\F$, and
performed $\Theta(n^2 k + \mu n)$ arithmetic operations in~$\F$.

Suppose next that $q \ge 3$. In this case one should begin, once again, by applying the process described in
Section~\ref{sec:nilpotent_blocks} to determine whether $A$ has more than $k$ nontrivial nilpotent blocks, in such a way
that this process fails with probability at most $\epsilon/2$ --- and in such a way that a pair of matrices $U_0 \in
\matgrp{\F}{k}{n}$ and $V_0 \in \matgrp{\F}{n}{k}$ has been discovered such that the minimal polynomial~$f_0$
of $A + V_0 \cdot U_0$ is not divisible by~$z^2$ --- and $f_0$ has been computed --- if the process has not
failed and $A$ has at most $k$ nontrivial nilpotent blocks.

Suppose now that $A$ has at most $k$ nontrivial nilpotent blocks, so that it is necessary to check whether
$\varphi_{k+1}$ has a monic irreducible factor in~$\F[z]$ that is different from~$z$. Let $c$ be a positive
integer, greater than or equal to two, and suppose that $c$ pairs of matrices $U_i \in \matring{\F}{k}{n}$ and
$V_i \in \matring{\F}{n}{k}$ are chosen uniformly and independently, for $1 \le i \le c$.

Since there are $q-1$ monic irreducible polynomials with degree in~$\F[z]$ with degree one, it follows by
part~(b) of Theorem~\ref{thm:probable_relationship} that $\gcd(f_1, f_2, \dots, f_c)$ has a monic irreducible
factor that is not also a factor of~$\varphi_{k+1}$ with probability at most
\[
 g_1(q, c) = (q-1) \cdot (1 - \rho_2(q, 1))^c.
\]

Since there are at most $\frac{q^d}{d}$ monic irreducible polynomials with degree~$d$ in~$\F[x]$, for $d \ge 2$, it
also follows that $\gcd(f_1, f_2, \dots, f_c)$ has a monic quadratic irreducible factor in~$\F[z]$ that is not also
a factor of~$\varphi_{k+1}$ with probability at most
\[
 g_2(q, c) = \frac{q^2}{2} \cdot (1 - \rho_2(q, 2))^c,
\]
and a monic cubic irreducible factor in~$\F[z]$ that is not a factor of~$\varphi_{k+1}$ with probability at most
\[
 g_3(q, c) = \frac{q^3}{3} \cdot (1 - \rho_2(q, 3))^c.
\]

Finally, the probability that $\gcd(f_1, f_2, \dots, f_c)$ has a monic irreducible factor in~$\F[z]$ with degree at least four,
that is not also a factor of~$\varphi_{k+1}$, is at most
\begin{align*}
 \sum_{d \ge 4} \frac{q^d}{d} \cdot (1 - \rho_2(q, d))^c
  &\le \sum_{d \ge 4} \frac{q^d}{4} \cdot (1 - \rho_2(q, d))^c \\
  &\le \sum_{d \ge 4} \frac{q^d}{4} \cdot \left( \frac{2}{q^d} \right)^c \\
  &= 2^{c-2} \sum_{d \ge 4} (q^{1-c})^d \\
  &= 2^{c-2} \frac{q^{4(1-c)}}{1 - q^{1-c}} \\
  &= 2^{c-2} \frac{q^{3(1-c)}}{q^{c-1}-1} = g_4(q, c).
\end{align*}

Consequently $\gcd(f_1, f_2, \dots, f_c)$ has a monic irreducible factor in~$\F[z]$, different from~$z$, that is not
also a factor of~$\varphi_{k+1}$, is at most
\[
 F(q, c) = g_1(q, c) + g_2(q, c) + g_3(q, c) + g_4(q, c)
\]
for the functions $g_1$, $g_2$, $g_3$ and~$g_4$ that are given above.

Straightforward calculations, aided by the use of a computer algebra system, confirm that $F(3, 4) < \frac{1}{2}
< F(3, 3)$, so that one can set $c=4$ when $q=3$ in order to ensure that the above process succeeds with
probability at least $\frac{1}{2}$ when $q=3$. Similarly, $F(q, 3) < \frac{1}{2} < F(q, 2)$, so that one can set
$c=3$, when $4 \le q \le 7$. Finally, $F(q, 2) < \frac{1}{2}$, so that one can set $c = 2$, when $q \ge 8$.

While it seems necessary to set $c=2$ for large field sizes as well, the probability of failure drops (for $c = 2$)
as $q = |\F|$ increases. Indeed, using the fact that $(1 - \rho_2(q, d)) \le 2 q^{-d}$ for $d \ge 1$, it is
straightforward to establish that $F(q, 2) \le 4 q^{-1}$ when $q \ge 5$.

As for the case that $q = 2$, $\Theta(\log_2 (\epsilon^{-1})$ independent trials should be used, in order to
ensure that a set of~$c$ pairs of matrices and corresponding minimal polynomials $f_1, f_2, \dots, f_c$ have
been found, so that the squarefree part of $\gcd(f_0, f_1, \dots, f_c)$ is a divisor of~$\varphi_{k+1}$ with
probability at least $1 - \epsilon/4$. Sufficiently many linear recurrences of the form shown at
line~\eqref{eq:Krylov_sequences} should be considered to ensure that all minimal polynomials of matrices
$A + V_i \cdot U_i$ have been correctly computed, with probability at least $1 - \epsilon/4$
as well, in order to ensure that this ``detection'' process fails with probability at most~$\epsilon$.


If it has been determined that $A$ has at most $k$ nontrivial invariant factors, so that $c+1$ pairs of matrices
$U_i \in \matgrp{\F}{k}{n}$ and $V_I \in \matgrp{\F}{n}{k}$, and corresponding minimal polynomials $f_i \in \F[z]$
have been accumulated, such that $\gcd(f_0, f_1, \dots, f_c) = \varphi_{k+1} \in \{1, z\}$, then it will also be
useful to compute polynomials $g_0, g_1, \dots, g_c \in \F[z]$ such that
\begin{equation}
\label{eq:gcd}
 g_0 f_0 + g_1 f_1 + \dots + g_c f_c = \varphi_{k+1}.
\end{equation}
Since $c$ is a constant, for any field size~$q$, the extended Euclidean algorithm can be applied to compute
these polynomials, such that each polynomial~$g_i$ has degree in~$O(n)$, at a cost that is dominated by
the cost of the rest of this process.

\subsection{Few Invariant Factors: Certification and Verification}

In order to \textbf{\emph{commit}} that $A \in \matring{\F}{n}$ has at most $k$ invariant factors, the prover
should send a sequence of $c+1$ pairs of matrices $U_i \in \matgrp{\F}{k}{n}$ and $V_i \in \matgrp{\F}{n}{k}$,
for $0 \le i \le c$, along with
\begin{itemize}
\item the minimal polynomial $f_i \in \F[z]$ of $A + V_i \cdot U_i$,
\item polynomials $g_i \in \F[z]$, for $0 \le i \le c$, each with degree in~$O(n)$, such that the
     equation at line~\eqref{eq:gcd} is also satisfied, where $\varphi_{k+1} \in \{ 1, z \}$.
\end{itemize}

As a \textbf{\emph{challenge}}, the verifier should select
\(
    \tau_3 = \lceil \log_{(2q^{-1}- q^{-2})^{-1}} (2(c+1) \tau_3 \epsilon^{-1}) \rceil
\)
pairs of vectors $u_i, v_i$ uniformly and independently from~$\columnspace{\F}{n}$, for $1 \le i \le \tau_3$,
and send these to the prover. The prover should then compute the minimal polynomial of each linearly recurrent
sequence
\begin{equation}
\label{eq:more_minimal_polynomials}
 u_i^T \cdot v_i,\, u_i^T (A + V_j \cdot U_j) \cdot v_i,\, u_i^T (A + V_j \cdot U_j)^2 v_i,\, \dots
\end{equation}
for $1 \le i \le \tau_3$ and $0 \le j \le c$, and should interact with the verifier to certify each of these
minimal polynomials --- choosing parameters in order to ensure that the verifier would discover an
incorrect minimal polynomial with probability at least $1 - \epsilon/(2 (c+1) \tau_3)$ if the prover tried to
provide one.

If the verifier can confirm that the minimal polynomial of each of the sequences shown at
line~\eqref{eq:more_minimal_polynomials} is divisible by the minimal polynomial~$f_i$ supplied by
the prover, for $0 \le i \le c$ and $1 \le j \le \tau_3$ and, furthermore, can confirm that
\[
 g_0 f_0 + g_1 f_1 + \dots g_c f_c \in \{ z, 1 \},
\]
then the verifier should \textbf{\emph{accept}}. Otherwise, the verifier should \textbf{\emph{reject}}.

Dumas, Kaltofen, Thom\'{e} and Villard~\cite{dum16} are primarily concerned with certification of the minimal
polynomial of a matrix over a large field, namely a field~$\F$ such that $|\F| \in \Omega(n)$. However their
observations about  small field computations can be extended, in a straightforward way, to establish the
following.

\begin{theorem}[Dumas, et. al.] It is possible for a prover to compute the minimal polynomial of a given
matrix $A \in \matring{\F}{n}$ and certify it using an interactive protocol that is perfectly complete and sound:
An incorrect minimal polynomial is accepted with probability at most~$\epsilon$ for any desired positive
constant~$\epsilon$. The prover selects $\Theta(n \log_2 n)$ values uniformly and independently from~$\F$
and performs $\Theta(n^2 \Mult{n} + \mu n \log_2 n)$ additional operations in~$\F$ while participating in this
process, where $\Mult{n}$ is the number of operations in~$\F$ required for an arithmetic operation in a
field extension whose degree over~$\F$ is logarithmic in~$n$. The verifier selects $O(n \log_2 n)$ values
uniformly and independently from~$\F$ and performs $O(n \Mult{n} + \mu)$ additional operations in~$\F$.
\end{theorem}

Soundness and perfect completeness of this protocol are easily established, provided that the protocol of
Dumas, Kaltofen, Thom\'{e} and Villard is used to ensure that any incorrect minimal polynomial of a linearly
recurrent sequence as shown at line~\eqref{eq:more_minimal_polynomials} would be detected by the
verifier with probability at least $1 - \frac{\epsilon}{2(c+1)\tau_3}$. The costs for the prover and
verifier, given in Theorem~\ref{thm:invariant_factors_summary}, follow from the fact that the cost to multiply 
$A + V_i \cdot U_i$ by a vector is in $\Theta(\mu + nk)$.

\subsection{Many Invariant Factors: Certification and\\ Verification}

In this case, the prover has detected a factor $\chi \in \F[z]$ of~$\varphi_{k+1}$ that is different from~$1$ or~$z$;
the prover should now \textbf{\emph{commit}} to this protocol by sending~$\chi$ to the verifier. If $\chi$ is
divisible by~$z^2$ then verifier should interact with the prover in order to confirm that $A$ has more
than $k$ nontrivial nilpotent blocks, as described in Section~\ref{sec:nilpotent_blocks}.

Otherwise, the verifier
should choose matrices $U_i \in \matgrp{\F}{k}{n}$ and~$V_i \in \matgrp{\F}{n}{k}$ uniformly and independently,
for $0 \le i \le \widetilde{\tau}-1$, where $\widetilde{\tau} = \lceil \log_{(1 - \rho_2(q, 1))^{-1}} ( 2 \epsilon^{-1})
\rceil$, and send these to the prover as a \textbf{\emph{challenge}}: If $\varphi_{k+1}$ is not divisible
by~$\chi$ then there will exist at least one pair of matrices $U \in \matgrp{\F}{k}{n}$ and
$V_i \in \matgrp{\F}{n}{k}$, such that the minimal polynomial of $A + V_i \cdot U_i$ is also not divisible
by~$\chi$, with probability at least $1 - \frac{\epsilon}{2}$.

Consequently, as a \textbf{\emph{response}} the prover should return vectors $u_i, v_i \in \columnspace{\F}{n}$
such that the
minimal polynomial of the linear recurrence 
\[
u_i^T v_i,\, u_i^T (A + V_i \cdot U_i) v_i,\, u_i^T (A + V_i \cdot U_i)^2 v_i,\, \dots
\]
is~$\chi$, for $0 \le i \le \widetilde{\tau}-1$. These can be obtained at sufficiently by following a process described, for
example, by Eberly~\cite{ebe00}, to compute vectors $u_i$ and~$v_i$ such the minimal polynomial of the above
linearly recurrent sequence is the minimal polynomial~$f_i$ of $A + V_i \cdot U_i$, and then replacing $v_i$
with $(f_i/\chi)(A + V_i \cdot U_i) v_i$.

The prover and verifier should apply the protocol of Dumas, Kaltofen, Thom\'{e}
and~Villard once again, in order to confirm this --- ensuring that any incorrect pair of vectors is accepted with 
probability at most $\epsilon/(2\widetilde{\tau})$, so that the total probability of failure is at most~$\epsilon$,
once again.

\section{Additional Problems}
\label{sec:additional_work}

Protocols certifying that a matrix $A \in \matring{\F}{n}$ is \emph{not} banded, with band width~$k$, or that the
rank (or displacement rank, for the types of this discussed in Section~\ref{sec:displacement_rank}) of a given
matrix exceeds~$k$, are also easily described as straightforward variants of those given in this report.

Additional properties allowing other ``superfast'' algorithms to be used might also be efficiently detected. For
example, one might be able to detect some cases when \textbf{\emph{nested dissection}} can be applied. The
detection of \textbf{\emph{Vandermonde-like}} and \textbf{\emph{Cauchy-like}} matrices might be of interest ---
and also might be more challenging than that of detecting Toeplitz-like matrices: The operator matrices for
Vandermonde-like and Cauchy-like matrices are defined using diagonal matrices whose entries can vary, and
one would need to discover these diagonal matrices as part of a detection process.

This report has focussed on the case where $k$ is extremely small. However, various ``superfast'' algorithms
would still be superior to a black box algorithm for larger~$k$ --- for example, for $k \le \sqrt{n}$. Detection and
conversion protocols that are effective, for larger~$k$, might therefore be of interest.

Protocols such that the parameter~$k$ is not selected by a client (or verifier), but is instead discovered by the
service provider (or prover), can be generally be obtained by modifying the protocols in this report in a
straightforward way.

Finally, black box algorithms are also used for various integer matrix computations, including computations
involving the \textbf{\emph{Smith form}} of a matrices. Protocols to decide and certify whether the number of
nontrivial elementary divisors of a given integer matrix would therefore be of interest.

\bibliographystyle{plain}

\end{document}